\documentclass[11pt]{article}
\usepackage{amsmath, amssymb, amsthm, bbm}
\usepackage{amsfonts, multirow, epsfig, subfig, floatrow}
\usepackage{graphicx, pdflscape, verbatim, enumerate, colortbl, setspace}
\usepackage{setspace, color}
\usepackage[normalem]{ulem}
\usepackage[left=2.3cm,top=2.3cm,right=2.3cm,bottom=2.3cm]{geometry}
\usepackage[authoryear,round,longnamesfirst]{natbib}
\usepackage[titletoc]{appendix}
\usepackage[dvipsnames,table,dvipsnames*, svgnames*, hyperref]{xcolor}

\newtheorem{Theorem}{Theorem}
\newtheorem{Lemma}{Lemma}

\newtheorem{Remark}{Remark}
\newcounter{Facts}

\definecolor{ZurichRed}{rgb}{1, 0, 0} 
\definecolor{Gray}{gray}{0.85}

\newcommand{\mA}{{\mathcal{A}}}

\newcommand{\mP}{\mathcal{P}}
\newcommand{\R}{\mathbb{R}}

\newcommand{\mX}{{\mathcal{X}}}

\newcommand{\Var}{{\rm Var}}

\newcommand{\tPhi}{{\tilde{\Phi}}}
\newcommand{\Qhat}{{\widehat{Q}}}

\newcommand{\half}{\frac{1}{2}}
\newcommand{\1}{\mathbbm{1}}

\title{Optimal Estimation of A Quadratic Functional\\ and Detection of Simultaneous Signals}
\author{T. Tony Cai and Xin Lu Tan}

\begin{document}

\maketitle

\begin{abstract}
Motivated by applications in genomics, this paper studies the problem of optimal estimation of a quadratic functional of two normal mean vectors, $Q(\mu, \theta) = \frac{1}{n}\sum_{i=1}^n\mu_i^2\theta_i^2$, with a particular focus on the case where both mean vectors are sparse. We propose optimal estimators of  $Q(\mu, \theta)$ for different regimes and establish the minimax rates of convergence over a family of parameter spaces. The optimal rates exhibit interesting phase transitions in this family.  The simultaneous signal detection problem is also considered under the minimax framework. It is shown that the proposed estimators for $Q(\mu, \theta)$ naturally lead to optimal testing procedures.  
\end{abstract}



\section{Introduction}

The problem of quadratic functional estimation occupies an important position in nonparametric and high-dimensional statistical inference. It is of significant interest in  its own right, and also has close connections to other important problems such as signal detection and construction of confidence balls. The focus so far has been on the one-sequence case. 
\cite{Bickel1988} showed that there is an interesting phase transition in the density estimation setting where the minimax rate of convergence is the usual parametric rate when the density function is sufficiently smooth, and is otherwise slower than the parametric rate. 

Under the Gaussian sequence model:
\begin{equation}
\label{gsminf}
Y_i = \theta_i + \sigma_n z_i, \qquad i = 1, 2, \ldots, 
\end{equation}
where $z_i \stackrel{iid}{\sim} N(0, 1)$, \cite{DonohoNussbaum1990}, \cite{Fan1991}, and \cite{Efromovich1996} further developed this theory for estimating $Q(\theta)=\sum\theta_i^2$ over quadratically convex parameter spaces such as hyperrectangles or Sobolev balls. 
The Gaussian sequence model \eqref{gsminf} is equivalent to the white noise with drift model and can be used to approximate nonparametric regression and density estimation models.  
\cite{CaiLow2005, CaiLow2006} considered minimax and adaptive estimation of the quadratic functional $Q(\theta)$ over parameter spaces that are not quadratically convex. It is shown that in such a setting optimal quadratic rules are often suboptimal and nonquadratic procedures may exhibit different phase transition phenomena than quadratic procedures. The results on estimating the quadratic functional $Q(\theta)$ have important implications on hypothesis testing and  construction of confidence balls. See, for example,  \cite{li1989}, \cite{dumbgen1998},  \cite{lepski1999},  \cite{ingster2003}, \cite{baraud2004}, \cite{Genovese2005},  and \cite{CaiLow2005, CaiLow2006b}.


Motivated by contemporary applications in genomics, we consider in the present paper estimation of the  functional  
\begin{equation}
Q(\mu, \theta) = \frac{1}{n}\sum_{i=1}^n\mu_i^2\theta_i^2 \label{Qmutheta}
\end{equation}
under the Gaussian two-sequence model,
\begin{equation}
X_i = \mu_i + \sigma z_i', \quad Y_i = \theta_i + \sigma z_i, \qquad i = 1,  \ldots, n, \label{twogsm}
\end{equation}
where $z_1', \ldots, z_n', \; z_1, \ldots, z_n\stackrel{iid}{\sim} N(0, 1)$ and $\sigma$ is the noise level. The goal is to optimally estimate the quadratic functional $Q(\mu, \theta)$ based on the observed data $(X_i, Y_i)$, $i=1, ..., n$. (Strictly speaking, $Q(\mu, \theta)$ is a quartic functional, but we will refer to it as a quadratic functional in the two-sequence case, as it is quadratic in $\mu$ given $\theta$, and vice versa.)
We are particularly interested in the case where both mean vectors $\mu = (\mu_1, \ldots, \mu_n)$ and $\theta = (\theta_1 ,\ldots, \theta_n)$ are sparse.

This estimation problem is motivated by the detection of simultaneous signals in genomics, where high-throughput technologies have generated a broad array of large-scale genome-wide datasets \citep{Schena1995, Lockhart2002, Puig2001}.  As the heterogeneous datasets provide distinct - but often complementary - views of biological systems, an integrative approach in data analysis is called for to obtain a coherent view of the underlying biology.  
As an example, it is of great interest to connect certain genotypes to specific phenotypic outcomes to infer causal relationship among genetic variation, expression and disease.  With regard to this, many genome-wide association studies (GWAS) have identified potential disease-associated SNPs, and a natural next step is to identify genes whose expression levels are regulated by the disease-associated SNPs.  
A possibly effective integrative approach exploits the potential overlap between SNPs associated with expression (expression SNPs) and the SNPs associated with disease (disease SNPs) for improved power in detecting gene-disease associations \citep{He2013}.  Recent findings also suggest overlapping SNPs between numerous human traits \citep{Sivakumaran2011} and disorders \citep{Cotsapas2011, Ripke2011}.  Thus combining GWAS statistics from two (or even multiple) disorders provides increased power for discovering genes associated with common biological mechanism, thereby informs on overlapping pathophysiological relationship between the disorders.  Other examples where detecting simultaneously occurring signals is of interest include the detection of shared DNA copy number variation across samples and meta-analysis of multiple linkage studies \citep{Zhang2010}.

In a simplified statistical framework, a problem of particular interest is detecting simultaneous signals under the Gaussian two-sequence model \eqref{twogsm}.  Specifically, let $\mu\star\theta = (\mu_1\theta_1, \ldots, \mu_n\theta_n)$ be the coordinate-wise product of $\mu$ and $\theta$.  For the mean vector $\mu$ (similarly, $\theta$), we say that there is a signal at location $i$ if $\mu_i\neq 0$ (similarly, $\theta_i\neq 0$).  Our goal is to detect the existence of simultaneous signal for $\mu$ and $\theta$, which corresponds to the presence of location $i$'s with $\mu_i\theta_i\neq 0$.  Equivalently, we want to distinguish between $\mu\star\theta = 0$ and $\mu\star\theta\neq 0$.  Of particular interest is the setting where the proportion of signals is small, and the signal strengths are relatively weak.  
This is indeed the setting in the gene-disease associations context, as only a small number of SNPs are expected to be associated with a disease or to regulate gene expression level.  Moreover, the association, if exists, is weak.  

As demonstrated in the single Gaussian sequence model setting considered in \cite{CaiLow2005}, 
the minimax hypothesis testing problem is closely connected to the minimax estimation theory.  
Our interest in detecting the existence of simultaneous signals for the unknown mean vectors $\mu$ and $\theta$ motivates the estimation of the quadratic functional of $(\mu, \theta)$ given in \eqref{Qmutheta}.
Note that $Q(\mu, \theta) = 0$ if and only if $\mu\star\theta = 0$.  Indeed, the study of estimation of $Q(\mu, \theta)$ turns out to highlight some important features of the testing problem. 
We emphasize that the two-sequence estimation and detection problem are not straightforward extension of the one-sequence case, and is interesting in its own right.

Our contribution is two-fold.  First, we propose optimal estimators of  $Q(\mu, \theta)$ over a family of parameter spaces to be introduced, and establish the minimax rates of convergence. It is shown that the optimal rate exhibits interesting phase transitions in this family.  Along with the establishment of the minimax rates of convergence, we explain the intuition behind the construction of the optimal estimators.  Second, we study the simultaneous signal detection problem under the minimax framework, and show that the proposed estimators for $Q(\mu, \theta)$ naturally lead to optimal testing procedures.  Thus, we bridge the gap between estimation and detection in the two-sequence case.  Our formulation of the simultaneous signal detection problem also provides an alternative view to that of \cite{Zhao2014}, where the problem is studied under the mixture model framework.  

The rest of the paper is organized as follows: Section \ref{twosample} considers estimation of the functional $Q(\mu, \theta)$ and establishes the minimax rates of convergence.  An application of the estimators of $Q(\mu, \theta)$ to the simultaneous signal detection problem is given in Section \ref{detection}.  Section \ref{simulation} complements our theoretical study with some simulation results, and we conclude the paper with a  discussion in Section \ref{discussion}.  Some additional results that are not included in the main text are given in Appendix \ref{general}. Proofs of some of the main results are given in Section~\ref{proofs}, with the rest relegated to Appendix \ref{prooftwosample} for the reason of space.

\section{Optimal Estimation of $Q(\mu, \theta)$} \label{twosample}
In this section, we consider the estimation of the quadratic functional $Q(\mu, \theta) = \frac{1}{n}\sum_{i=1}^n\mu_i^2\theta_i^2$ of two sparse normal mean vectors $\mu = (\mu_1, \ldots, \mu_n)$ and $\theta = (\theta_1, \ldots, \theta_n)$ under the Gaussian two-sequence model \eqref{twogsm}.
An additional constraint is also imposed on the number of coordinates that are simultaneously nonzero for both mean vectors.  The noise level $\sigma$ in model \eqref{twogsm} is assumed to be known.  Estimation of the noise level, $\sigma$, is relatively easy under the sparse sequence model \eqref{twogsm} and will be discussed in Section \ref{simulation}.

We begin by introducing some notation that will be used throughout the paper.
Given a vector $\theta = (\theta_1, \ldots, \theta_n)$, we denote by $\|\theta\|_0 = \text{Card}(\{i: \theta_i\neq 0\})$ the $\ell_0$-quasi-norm of $\theta$, $\|\theta\|_2 = \sqrt{\sum_{i=1}^n\theta_i^2}$ its $\ell_2$-norm, and $\|\theta\|_\infty = \max_{1\leq i\leq n} |\theta_i|$ its $\ell_\infty$-norm.  For any real number $a$ and $b$, set $a\wedge b = \min\{a, b\}, a\vee b = \max\{a, b\}$ and $a_+ = a\vee 0$.  Throughout, the notation $a_n\asymp b_n$ means that there exists some numerical constants $c$ and $C$ such that $c\leq \frac{a_n}{b_n}\leq C$ when $n$ is large.  By ``numerical constants" we usually mean constants that might depend on the characteristics of the problem but whose specific values are of little interest to us.  The precise values of the numerical constants $c$ and $C$ may also vary from line to line.

Adopting an asymptotic framework where the vector size $n$ is the driving variable, we parameterize the signal strength, sparsity, and simultaneous sparsity of $\mu$ and $\theta$ as functions of $n$.  Specifically, we consider the family of parameter spaces
\begin{align}
\Omega(\beta, \epsilon, b)  &= \{(\mu, \theta)\in\R^n\times\R^n: \|\mu\|_0\leq k_n, \|\mu\|_\infty\leq s_n, \|\theta\|_0\leq k_n, \|\theta\|_\infty\leq s_n, \nonumber\\
&\qquad\qquad\qquad\qquad\qquad \|\mu\star\theta\|_0\leq q_n\}, \label{twoSpace}
\end{align}
indexed by three parameters $\beta, \epsilon$, and $b$.  We have the sparsity parameterization
\begin{equation}
k_n = n^\beta, \qquad\qquad 0<\beta<\half, \label{sparsity}
\end{equation}
the simultaneous sparsity parameterization
\begin{equation}
q_n = n^\epsilon, \qquad\qquad 0<\epsilon\leq\beta, \label{simsparsity}
\end{equation}
and the signal strength parametrization
\begin{equation}
s_n = n^b, \textcolor{ZurichRed}\qquad\qquad b\in\R. \label{magnitude}
\end{equation}
In other words, $\Omega(\beta, \epsilon, b)$ is the collection of vector pairs $(\mu, \theta)\in\R^n\times\R^n$, where both $\mu$ and $\theta$ have at most $k_n$ nonzero entries, each entry is bounded in its magnitude by $s_n$, and the number of simultaneous nonzero entries for $\mu$ and $\theta$ is at most $q_n$.  In principle, $\beta$ can take any value between 0 and 1.  We are primarily interested in the estimation problem for the range $0<\beta<\half$, as it is well-known that this corresponds to the case of rare signals \citep{DonohoJin2004}.  Also, even though we parametrize the signal strength at the algebraic order $s_n = n^b$, throughout we will make remark on the estimation result for $s_n$ of order $\sqrt{\log n}$, since this is an interesting region in the one-sequence signal detection problem \citep{DonohoJin2004}.

Our goal is to derive the minimax rate of convergence for $Q(\mu, \theta)$ over $\Omega(\beta, \epsilon, b)$:
\[R^*(n, \Omega(\beta, \epsilon, b)) = \inf_{\Qhat}\sup_{(\mu, \theta)\in\Omega(\beta, \epsilon, b)} E_{(\mu, \theta)}(\Qhat-Q(\mu, \theta))^2.\]
We will show that $R^*(n, \Omega(\beta, \epsilon, b))$ satisfies
\begin{equation}
R^*(n, \Omega(\beta, \epsilon, b)) \asymp \gamma_n(\beta, \epsilon, b), \label{minimax}
\end{equation}
where $\gamma_n(\beta, \epsilon, b)$ is a function of $n$ indexed by $\beta, \epsilon$ and $b$.  There are two main tasks in establishing the minimax rate of convergence.  For each triple $(\beta, \epsilon, b)$ satisfying $0<\epsilon\leq\beta<\half$ and $b\in\R$, we 
\begin{enumerate}[(a)]
\item construct an estimator $\Qhat^*$ that satisfies
\[\sup_{(\mu, \theta)\in\Omega(\beta, \epsilon, b)}E_{(\mu, \theta)}(\Qhat^*-Q(\mu, \theta))^2 \leq C\gamma_n(\beta, \epsilon, b),\]
\item and show that 
\[R^*(n, \Omega(\beta, \epsilon, b)) \geq c\gamma_n(\beta, \epsilon, b),\]
\end{enumerate}
where $C$ and $c$ are numerical constants that depend only on $\beta, \epsilon, b$, and $\sigma$.  Combining the upper bound derived in task (a) and the lower bound derived in task (b) yields the minimax rate of convergence \eqref{minimax}.  In this case, we say that the estimator $\Qhat^*$ attains the minimax rate of convergence over the parameter space $\Omega(\beta, \epsilon, b)$.

Interestingly, the estimation problem exhibits different phase transitions for the minimax rate $\gamma_n(\beta, \epsilon, b)$ in three regimes: the \emph{sparse} regime where $0<\epsilon<\frac{\beta}{2}$, the \emph{moderately dense} regime where $\frac{\beta}{2}\leq\epsilon\leq\frac{3\beta}{4}$, and the \emph{strongly dense} regime where $\frac{3\beta}{4}<\epsilon\leq\beta$.  Collectively, we call $\frac{\beta}{2}\leq\epsilon\leq\beta$ the \emph{dense} regime.  In the sparse regime, simultaneous signal is sparse in the sense that $q_n\ll\sqrt{k_n}$, while in the dense regime, simultaneous signal is dense in the sense that $q_n\gg\sqrt{k_n}$.  This is analogous to the terminology used in the one-sequence model, where signal is called sparse if $0<\beta<\half$ ($k_n\ll\sqrt{n}$), and dense if $\half\leq\beta\leq1$ ($k_n\gg\sqrt{n}$).  The key distinction is that, in the two-sequence case, the sparseness or denseness is used to describe the relationship between simultaneous sparsity $q_n$ and sparsity $k_n$, as opposed to between $k_n$ and the vector size $n$.  We also remark that our use of the terminology is not superficial --- a detailed analysis of lower bound and upper bound for the estimation problem does reveal intimate connection to the corresponding regimes in the one-sequence case.

Intuitively, when $b$ is very small (i.e., signal is very weak), we are better off estimating $Q(\mu, \theta)$ by 
\begin{equation}
\Qhat_0 = 0, \label{estimator0}
\end{equation}
since any attempt to estimate $Q(\mu, \theta)$ will incur a greater estimation risk.  On the other hand, when $b$ is sufficiently large (i.e., signal is strong), it is desirable to estimate $Q(\mu, \theta)$ based on the observed data $(X_i, Y_i)$, $i = 1, \ldots, n$.  With a slight abuse of terminology, we say that the signal is weak if it corresponds to the region where $\Qhat_0$ is optimal, and we say that the signal is strong otherwise.  We construct two estimators of $Q(\mu, \theta)$ that respectively attain the minimax rates of convergence over the sparse and dense regimes when the signal is sufficiently large in Sections~\ref{sparseest} and \ref{denseest}.

Note that it is possible to generalize our parametrization to the case where $\mu$ and $\theta$ have different levels of both sparsity and signal strengths.  This amounts to estimating $Q(\mu, \theta)$ over the parameter space
\begin{align}
\Omega(\alpha, \beta, \epsilon, a, b)  &= \{(\mu, \theta)\in\R^n\times\R^n: \|\mu\|_0\leq j_n, \|\mu\|_\infty\leq r_n, \|\theta\|_0\leq k_n, \|\theta\|_\infty\leq s_n, \nonumber\\
&\qquad\qquad\qquad\qquad\qquad\|\mu\star\theta\|_0\leq q_n\}, \label{compSpace}
\end{align}
where $j_n = n^\alpha, k_n = n^\beta, q_n = n^\epsilon$ with $0<\epsilon\leq\alpha\wedge\beta<\half$, and $r_n = n^a, s_n = n^b$ with $a, b\in\R$.  In this section, however, we will focus on the simplest case where $j_n = k_n = n^\beta$ and $r_n = s_n = n^b$, since the technical analysis is similar to that for the more general case  \eqref{compSpace} but less tedious.  We did derive the minimax rate of convergence for the case where $j_n = k_n = n^\beta$ but $r_n$ and $s_n$ are allowed to differ.  As the phase transitions for the minimax rates of convergence in this case are much more sophisticated but also are less easily digestible, we opt to defer its presentation to Appendix~\ref{generalEst}.   The analysis for the general case \eqref{compSpace} where no constraint is imposed on either the sparsity or signal strength of $\mu$ and $\theta$ follows similarly, provided that the magnitude of the simultaneous sparsity $\epsilon$ is compared to $\alpha$ if $a\geq b$, and to $\beta$ if $b\geq a$, for the determination of sparse and dense regimes.

\subsection{Estimation in the Sparse Regime}\label{sparseest}
We begin with the estimation of $Q(\mu, \theta) = \frac{1}{n}\sum\mu_i^2\theta_i^2$ over the parameter space $\Omega(\beta, \epsilon, b)$ in the sparse regime, where $q_n$ is calibrated as in expression \eqref{simsparsity} with $0<\epsilon<\frac{\beta}{2}$.

To construct an optimal estimator for $Q(\mu, \theta)$, we base our intuition on the estimation of the quadratic functional $Q(\theta) = \frac{1}{n}\sum\theta_i^2$, in the case where we only have one sequence of observations $Y_i$, $i = 1, \ldots, n$, from model \eqref{twogsm}.  Consider the family of parameter spaces indexed by $k_n = n^\beta, 0<\beta<1$ and $s_n = n^b, b\in\R$:
\begin{equation}
\Theta(\beta, b)  = \{\theta\in\R^n: \|\theta\|_0\leq k_n, \|\theta\|_\infty\leq s_n\}. \label{oneSpace}
\end{equation}
That is, $\Theta(\beta, b)$ is the collection of vectors in $\R^n$ that has at most $k_n$ nonzero entries uniformly bounded in magnitude by $s_n$.  It can be shown that for $0<\beta<\half$, the minimax rate of convergence for $Q(\theta)$ over $\Theta(\beta, b)$ satisfies
\begin{equation}
R^*(n, \Theta(\beta, b)) := \inf_{\Qhat}\sup_{\theta\in\Theta(\beta, b)}E_\theta(\Qhat-Q(\theta))^2 \asymp \gamma_n(\beta, b), \label{minimaxOne}
\end{equation}
where
\begin{equation}
\gamma_n(\beta, b) = \left\{
\begin{array}{ll}
n^{2\beta+4b-2} &\text{if }b\leq 0, \\
n^{2\beta-2}(\log n)^2 &\text{if }0 < b \leq \frac{\beta}{2}, \\
n^{\beta+2b-2} &\text{if }b>\frac{\beta}{2}. \label{onesamplerateS}
\end{array}\right.
\end{equation}
Moreover, the minimax rate of convergence when $s_n = \sigma\sqrt{d\log n}$ for some $d>0$ satisfies 
\[\inf_{\Qhat}\sup_{\theta: \|\theta\|_0\leq k_n, \|\theta\|_\infty\leq \sigma\sqrt{d\log n}}E_\theta(\Qhat-Q(\theta))^2 \asymp n^{2\beta-2}(\log n)^2.\]
Thus, the phase transition of $\gamma_n(\beta, b)$ from $b\leq 0$ to $b>0$ is smooth.  The special interest on the signal strength along the order of $\sqrt{\log n}$ has its root in the one-sequence signal detection problem, which we will discuss in more details in Section~\ref{detection}.

When $0<\beta<\half$, we have $k_n\ll\sqrt{n}$.  Thus, we anticipate only very few coordinates of $\theta$ to be nonzero.  If, in addition, $b<0$, then the signal is both rare and weak, and one can do no better than simply estimating $Q(\theta)$ by $\Qhat_0 = 0$.  Nonetheless, when $b>0$, signal is rare but sufficiently strong, and the estimator
\begin{equation}
\Qhat_1 = \frac{1}{n}\sum_{i=1}^n[(Y_i^2-\sigma^2\tau_n)_+ - \theta_0], \qquad\text{where } \theta_0 := E_0(Y_i^2-\sigma^2\tau_n)_+ \label{estimator1}
\end{equation}
which performs coordinate-wise thresholding on $Y_i^2$ with choice of tuning parameter $\tau_n = 2\log n$ is optimal.  Note that each term $\theta_i^2$ is estimated independently by $(Y_i^2-\sigma^2\tau_n)_+ - \theta_0$, since the sparsity pattern is unstructured.  The estimator \eqref{estimator1} involves a thresholding step, $(Y_i^2-\sigma^2\tau_n)_+$, for denoising, and a de-bias step by subtracting $\theta_0$ from the thresholded term so that we estimate the zero coordinates of $\theta$ unbiasedly.  This is important because the proportion of zero entries in this case is relatively large, and a biased estimator for these coordinates will unnecessarily inflates the estimation risk.  When $s_n = \sigma\sqrt{d\log n}$ for some $d>0$, we are indifferent in terms of estimation, since both $\Qhat_0$ and $\Qhat_1$ attains the minimax rate of convergence.

We now return to the estimation of $Q(\mu, \theta)$ in the two-sequence case, where $0<\epsilon<\frac{\beta}{2}$ and $0<\beta<\half$.  In this case, $k_n\ll\sqrt{n}$, so the signal of individual sequences is rare.  Moreover, the simultaneous sparsity $q_n \ll \sqrt{k_n}$, implying that knowledge about whether $\mu_i$ is nonzero does not entail much about whether $\theta_i$ is nonzero (and vice versa).  This motivates the estimator
\begin{equation}
\Qhat_2 = \frac{1}{n}\sum_{i=1}^n[(X_i^2-\sigma^2\tau_n)_+ - \mu_0][(Y_i^2-\sigma^2\tau_n)_+ - \theta_0], \qquad \mu_0 = \theta_0 := E_0(Y_i^2-\sigma^2\tau_n)_+ \label{estimator2}
\end{equation}
in the case of sufficiently strong signal, where the threshold level $\tau_n = \log n$.  The construction of $\Qhat_2$ is a straightforward extension of the construction of $\Qhat_1$: each term $\mu_i^2\theta_i^2$ is estimated independently by the product $[(X_i^2-\sigma^2\tau_n)_+ - \mu_0][(Y_i^2-\sigma^2\tau_n)_+ - \theta_0]$.  Since $q_n \ll \sqrt{k_n}$, following our previous argument, thresholding $X_i^2$ and $Y_i^2$ independently seems reasonable.  

We now present a theorem on the upper bound of the mean squared error of $\Qhat_2$.
\begin{Theorem}[Sparse Regime: Upper Bound]\label{two-inf-sparse-up}
For $b>0$, the estimator $\Qhat_2$ as in \eqref{estimator2} with $\tau_n = \log n$ satisfies
\begin{equation}
\sup_{(\mu, \theta)\in\Omega(\beta, \epsilon, b)}E_{(\mu, \theta)}(\Qhat_2-Q(\mu, \theta))^2 \leq C\Big[n^{2\epsilon+4b-2}(\log n)^2 + n^{\epsilon+6b-2}\Big].  \label{Q3rate}
\end{equation}
\end{Theorem}

Straightforward calculation shows that for the estimator $\Qhat_0 = 0$,
\begin{align}
\sup_{(\mu, \theta)\in \Omega(\beta, \epsilon, b)} E_{(\mu, \theta)}(\Qhat_0-Q(\mu, \theta))^2 &= \sup_{(\mu, \theta)\in \Omega(\beta, \epsilon, b)}\bigg(\frac{1}{n}\sum_{i=1}^n\mu_i^2\theta_i^2\bigg)^2 \nonumber\\
&= q_n^2s_n^8n^{-2} = n^{2\epsilon+8b-2}, \label{Q0rate}
\end{align}
for $0<\epsilon\leq\beta<\half$ and $b\in\R$.  We now show that the combination of $\Qhat_0$ (when $b<0$) and $\Qhat_2$ (when $b\geq 0$) is optimal, by providing a matching lower bound.

\begin{Theorem}[Sparse Regime: Lower Bound]\label{two-inf-sparse}
Let $0<\epsilon<\frac{\beta}{2}$ and $0<\beta<\half$.  Then
\[R^*(n, \Omega(\beta, \epsilon, b)) \geq c\gamma_n(\beta, \epsilon, b),\]
where
\begin{equation}
\gamma_n(\beta, \epsilon, b) = \left\{
\begin{array}{ll}
n^{2\epsilon+8b-2} &\text{if }b\leq 0, \\
n^{2\epsilon+4b-2}(\log n)^2 &\text{if }0<b\leq\frac{\epsilon}{2}, \\
n^{\epsilon+6b-2} &\text{if }b>\frac{\epsilon}{2}. \label{srate}
\end{array}\right.
\end{equation}
\end{Theorem}

Crucial to the derivation of lower bound is the Constrained Risk Inequality (CRI) given in \cite{BrownLow1996}.  To apply CRI, it suffices to construct two priors supported on $\Omega(\beta, \epsilon, b)$ that have small chi-square distance but a large difference in the expected values of the resulting quadratic functionals.  The cases $b\leq\frac{\epsilon}{2}$ and $b>\frac{\epsilon}{2}$ correspond to choices of distinct pairs of priors.  For $b>\frac{\epsilon}{2}$, the CRI boils down to the standard technique of inscribing a hardest hyperrectangle, with the Bayes risk for a simple prior supported on the hyperrectangle being a lower bound for the minimax risk.  Nevertheless, the case $b\leq\frac{\epsilon}{2}$ requires the use of a rich collection of hyperrectangles and a mixture prior which mixes over the vertices of the hyperrectangles in this collection.  Mixing increases the difficulty of the Bayes estimation problem and is needed here to attain a sharp lower bound.  

\begin{Remark}{\rm
Combining \eqref{Q3rate}, \eqref{Q0rate} and \eqref{srate}, we see that when $0<\epsilon<\frac{\beta}{2}$ and $0<\beta<\half$, $\Qhat_2$ attains the optimal rate of convergence over $\Omega(\beta, \epsilon, b)$ when $b>0$.  In contrast, $\Qhat_0$ attains the optimal rate of convergence over $\Omega(\beta, \epsilon, b)$ when $b\leq 0$.  
}
\end{Remark}

\begin{Remark}{\rm
Interestingly, there is no dependence on $\beta$ in the minimax rate of convergence $\gamma_n(\beta, \epsilon, b)$ in the sparse regime, except for the requirement that $0<\epsilon<\frac{\beta}{2}$.  
}
\end{Remark}

\subsection{Estimation in the Dense Regime}
\label{denseest}
We now consider estimating $Q(\mu, \theta)$ in the dense regime, where $q_n$ is calibrated as in expression \eqref{simsparsity} with $\frac{\beta}{2}\leq\epsilon\leq\beta$.  The dense regime is subdivided into two cases: the moderately dense case with $\frac{\beta}{2}\leq\epsilon\leq\frac{3\beta}{4}$ and the strongly dense case with $\frac{3\beta}{4}<\epsilon\leq\beta$.  

In the dense regime, the estimator $\Qhat_2$ defined in \eqref{estimator2} is suboptimal, as the thresholding step in both $X_i^2$ and $Y_i^2$ ends up thresholding too many coordinates when the signal is weak.  Note that the simultaneous sparsity $q_n \gg \sqrt{k_n}$ suggests that for each coordinate $i$ with $\mu_i\neq 0$, it is usually the case that $\theta_i\neq 0$, and vice versa.  Therefore, it is no longer reasonable to perform thresholding on $X_i^2$ and $Y_i^2$ independently.  The additional knowledge of relatively high proportion of simultaneous nonzero entries suggests that whenever we observe a large value of $X_i^2$ (an implication of $\mu_i\neq 0$), then even if $Y_i^2$ is small, we should still estimate $\mu_i^2\theta_i^2$ rather than setting it equals zero.  The same reasoning applies to the case where $X_i^2$ is small but $Y_i^2$ is large.  

To construct an optimal estimator in the dense regime, we again borrow some intuition from the estimation of the quadratic functional $Q(\theta) = \frac{1}{n}\sum\theta_i^2$ in the one-sequence case.  We consider the family of parameter spaces given in \eqref{oneSpace}, but for $\half\leq\beta<1$.  The minimax rate of convergence once again satisfies \eqref{minimaxOne}, but with
\begin{equation}
\gamma_n(\beta, b) = \left\{
\begin{array}{ll}
n^{2\beta+4b-2} &\text{if }b\leq \frac{1-2\beta}{4}, \\
n^{-1} &\text{if }\frac{1-2\beta}{4} < b \leq \frac{1-\beta}{2}, \\
n^{\beta+2b-2} &\text{if }b>\frac{1-\beta}{2}. \label{onesamplerateD}
\end{array}\right.
\end{equation}
When $\half\leq\beta<1$, we have $k_n\gg\sqrt{n}$, meaning that many coordinates of $\theta$ is nonzero.  The characterization of weak and strong signal is no longer $b<0$ versus $b\geq 0$ as in the case of $0<\beta<\half$, but $b\leq \frac{1-2\beta}{4}$ versus $b> \frac{1-2\beta}{4}$.   That is, given the same signal strength $b$, the vast number of nonzero coordinates of $\theta$ when $k_n\gg\sqrt{n}$ collectively represents stronger signal as compared to the case when $k_n\ll\sqrt{n}$.  Thus, the threshold of ``strong" signal as encoded by $b$ is lowered when $k_n\gg \sqrt{n}$.  It is not surprising that for the range of weak signal $b\leq\frac{1-2\beta}{4}$, the estimator $\Qhat_0 = 0$ is optimal.  On the other hand, when $b>\frac{1-2\beta}{4}$, the optimal estimator for $Q(\theta)$ is the unbiased estimator
\begin{equation}
\Qhat_3= \frac{1}{n}\sum_{i=1}^n(Y_i^2-\sigma^2). \label{estimator3}
\end{equation}

An optimal estimator is often one that strikes an appropriate balance between bias and variance in its mean squared error.  The estimators $\Qhat_0$ and $\Qhat_3$ represent two extremities in terms of bias-variance tradeoff.  We see that $\Qhat_0$ that is optimal for exceedingly weak signal has zero variance, while $\Qhat_3$ that is optimal for sufficiently strong signal has zero bias.  Due to the denseness of nonzero coordinates when $k_n\gg\sqrt{n}$, one could not afford to introduce bias to the estimator in the hope of achieving smaller variance.  Without additional information about the sparsity structure, the unbiased estimator $\Qhat_3$ is necessary for optimal estimation of $Q(\theta)$.

We now return to the two-sequence setting for the estimation of $Q(\mu, \theta)$, for the case $\frac{\beta}{2}\leq\epsilon\leq\beta$ and $0<\beta<\half$.  Although the signal for individual sequences is sparse ($k_n\ll\sqrt{n}$), the simultaneous signal is dense in the sense that $q_n\gg \sqrt{k_n}$.  The intuition garnered from the one-sequence case motivates the following estimator:
\begin{equation}
\Qhat_4 = \frac{1}{n}\sum_{i=1}^n\big[(X_i^2-\sigma^2)(Y_i^2-\sigma^2)\1(X_i^2\vee Y_i^2>\sigma^2\tau_n)-\eta\big],  \label{estimator4}
\end{equation}
where
\[\eta = E_{(0, 0)}[(X_i^2-\sigma^2)(Y_i^2-\sigma^2)\1(X_i^2\vee Y_i^2>\sigma^2\tau_n)].\]
From $\Qhat_4$, we see that each term $\mu_i^2\theta_i^2$ is estimated unbiasedly (modulo $\eta$) by $(X_i^2-\sigma^2)(Y_i^2-\sigma^2)$ whenever at least one of $X_i^2$ and $Y_i^2$ is sufficiently large.  This is in accordance with our previous argument that estimation should be done whenever we have at least one large value of $X_i^2$ or $Y_i^2$.  The threshold $\tau_n$ is a tuning parameter whose value is yet to be determined during the analysis of the mean squared error of $\Qhat_4$, though it turns out that $\tau_n = c\log n$ for any $c\geq 4$ attains the optimal rate of convergence.  The subtraction of $\eta$ from $(X_i^2-\sigma^2)(Y_i^2-\sigma^2)\1(X_i^2\vee Y_i^2>\sigma^2\tau_n)$ is needed because the majority of coordinates $i$ has $\mu_i = \theta_i = 0$.  A biased estimator for these coordinates unavoidably inflates the estimation risk.  Nevertheless, due to the rarity of nonzero coordinates in individual sequences, the naive unbiased estimator 
\begin{equation}
\Qhat_5 = \frac{1}{n}\sum_{i=1}^n(X_i^2-\sigma^2)(Y_i^2-\sigma^2)
\end{equation}
is not optimal, as one would have expected.  A thresholding step $\1(X_i^2\vee Y_i^2>\sigma^2\tau_n)$ is needed to guard against estimating entries with $\mu_i = \theta_i = 0$ with noise.

Note that $\Qhat_2$ defined in \eqref{estimator2} can be written as
\[\frac{1}{n}\sum_{i=1}^n[(X_i^2-\sigma^2\tau_n)\1(X_i^2>\sigma^2\tau_n) - \mu_0][(Y_i^2-\sigma^2\tau_n)\1(Y_i^2>\sigma^2\tau_n) - \theta_0].\]
Compare this expression with $\Qhat_4$, we see that when both $X_i^2$ and $Y_i^2$ are large, the term $\mu_i^2\theta_i^2$ is roughly estimated as $(X_i^2-\sigma^2\tau_n)(Y_i^2-\sigma^2\tau_n)$.  Moreover, $(X_i^2-\sigma^2\tau_n)(Y_i^2-\sigma^2\tau_n)$ is a biased estimator of $\mu_i^2\theta_i^2$ when $\tau_n>1$.

We present an upper bound on the mean squared error of $\Qhat_4$ in the following theorem.
\begin{Theorem}[Dense Regime: Upper Bound]\label{two-inf-dense-up}
For $b>0$, the estimator $\Qhat_4$ as in \eqref{estimator4} with $\tau_n = 4\log n$ satisfies
\begin{equation}
\sup_{(\mu, \theta)\in\Omega(\beta, \epsilon, b)}E_{(\mu, \theta)}(\Qhat_4-Q(\mu, \theta))^2 \leq C\max\Big\{n^{2\epsilon-2}(\log n)^4, n^{\epsilon+6b-2}, n^{\beta+4b-2}\Big\}. \label{Q4rate}
\end{equation}
\end{Theorem}

We now provide a matching lower bound to complement the upper bound in the dense regime.  
\begin{Theorem}[Dense Regime: Lower Bound]\label{two-inf-dense}
Let $\frac{\beta}{2}\leq\epsilon\leq\beta$ and $0<\beta<\half$.  Then
\[R^*(n, \Omega(\beta, \epsilon, b)) \geq c\gamma_n(\beta, \epsilon, b),\]
where
\begin{equation}
\gamma_n(\beta, \epsilon, b) = \left\{
\begin{array}{ll}
n^{2\epsilon+8b-2} &\text{if }b\leq 0, \\
n^{2\epsilon-2}(\log n)^4 &\text{if }0<b\leq\frac{2\epsilon-\beta}{4}, \\
n^{\beta+4b-2} &\text{if }\frac{2\epsilon-\beta}{4}<b\leq\frac{\beta-\epsilon}{2}, \\
n^{\epsilon+6b-2} &\text{if }b>\frac{\beta-\epsilon}{2}, \label{drate}
\end{array}\right.
\end{equation}
when $\frac{\beta}{2}\leq\epsilon\leq\frac{3\beta}{4}$, and
\begin{equation}
\gamma_n(\beta, \epsilon, b) = \left\{
\begin{array}{ll}
n^{2\epsilon+8b-2} &\text{if }b\leq 0, \\
n^{2\epsilon-2}(\log n)^4 &\text{if }0<b\leq\frac{\epsilon}{6}, \\
n^{\epsilon+6b-2} &\text{if }b>\frac{\epsilon}{6}. \label{udrate}
\end{array}\right.
\end{equation}
when $\frac{3\beta}{4}<\epsilon\leq\beta$.
\end{Theorem}

The minimax rates of convergence display different phase transitions within two subdivisions of the dense regime.  In the moderately dense regime where $\frac{\beta}{2}\leq\epsilon\leq\frac{3\beta}{4}$, there are phase transitions at $b = \frac{2\epsilon-\beta}{4}$ and $b = \frac{\beta-\epsilon}{2}$, given in \eqref{drate}.  Note that $\frac{2\epsilon-\beta}{4}\leq\frac{\beta-\epsilon}{2}$ if and only if $\epsilon\leq\frac{3\beta}{4}$.  In the strongly dense regime where $\epsilon>\frac{3\beta}{4}$, the phase $\frac{2\epsilon-\beta}{4}<b\leq\frac{\beta-\epsilon}{2}$ is non-existent, and we only have one intermediate phase $0<b\leq\frac{\epsilon}{6}$, given in \eqref{udrate}.

We establish the lower bound by constructing least favorable priors and applying CRI.  Except for the rate $n^{\epsilon+6b-2}$ which is obtained through the inscription of a hardest hyperrectangle, all other cases require some forms of mixing over the vertices of a rich collection of hyperrectangles.  

\begin{Remark}{\rm
Combining \eqref{Q0rate}, \eqref{Q4rate}, \eqref{drate}, and \eqref{udrate}, we see that for the parameter space $\Omega(\beta, \epsilon, b)$ with $\frac{\beta}{2}\leq\epsilon\leq\beta<\half$, $\Qhat_4$ attains the minimax rate of convergence when $b>0$.  In contrast, $\Qhat_0 = 0$ attains the minimax rate of convergence when $b\leq 0$.
}
\end{Remark}

\begin{Remark}{\rm
Similar to the sparse regime, there is no dependence on $\beta$ in the minimax rate of convergence $\gamma_n(\beta, \epsilon, b)$ in the strongly dense regime, except for the requirement that $\frac{3\beta}{4}<\epsilon\leq\beta$.  In contrast, $\gamma_n(\beta, \epsilon, b)$ does depend explicitly on $\beta$ in the moderately dense regime $\frac{\beta}{2}\leq\epsilon\leq\frac{3\beta}{4}$.
}
\end{Remark}

\begin{Remark}{\rm
For either the sparse or dense regime, i.e., $0<\epsilon\leq\beta$, the minimax rate of convergence of $Q(\mu, \theta)$ over $\Omega(\beta, \epsilon, b)$ diverges to infinity when $b>\frac{2-\epsilon}{6}$.  Hence when the signal strength is too large, it is impossible to consistently estimate $Q(\mu, \theta)$.
}
\end{Remark}

\begin{Remark}{\rm
For either the sparse or dense regime, i.e., $0<\epsilon\leq\beta$, when $r_n = s_n = \sigma\sqrt{d\log n}$ for some $d>0$, we have 
\[\inf_{\Qhat}\sup_{\substack{(\mu, \theta): \|\mu\|_0\leq k_n, \|\theta\|_0\leq k_n, \\\|\mu\|_\infty\leq \sigma\sqrt{d\log n}, \|\theta\|_\infty\leq \sigma\sqrt{d\log n}}}E_{(\mu, \theta)}(\Qhat-Q(\mu, \theta))^2 \asymp n^{2\epsilon-2}(\log n)^4.\]
The minimax rate of estimation is attained by $\Qhat_0, \Qhat_2$ for $0<\epsilon\leq\beta$ and by $\Qhat_4$ for $\frac{\beta}{2}\leq\epsilon\leq\beta$.
}
\end{Remark}

\subsection{Phase Transitions in the Minimax Rates of Convergence}\label{phaseplot}
We see from Sections~\ref{sparseest} and \ref{denseest} that within each regime, the minimax rates of convergence exhibit several phase transitions.  In addition, each transition is governed by a change in the relative magnitudes of the sparsity parameter $\beta$, the simultaneous sparsity parameter $\epsilon$, and the signal strength parameter $b$.  In fact, it is the way phase transitions occur within each regime that characterizes the regime itself.  Furthermore, the phase transitions actually display ``continuity" across the boundaries of different regimes.

To depict what we meant graphically, first note that from Sections~\ref{sparseest} and \ref{denseest}, the minimax rate of convergence
\begin{equation}
\gamma_n(\beta, \epsilon, b) \asymp n^{r(\beta, \epsilon, b)},
\end{equation}
modulo a factor involving $\log n$ when applicable.  In Figure~\ref{phase}, we plot the rate exponent $r(\beta, \epsilon, b)$ against $b$ for the sparse, moderately dense, and strongly dense regimes.  

Specifically, fixing $\beta = 0.45$, we plot $r(\beta, \epsilon, b)$ against $b$ for a range of $\epsilon$ values in $(0, \beta)$.  The left panel of Figure~\ref{phase} provides a continuum view of $r(\beta, \epsilon, b)$, as $\epsilon$ increases from 0 to $\beta$.  Each piecewise straight line corresponds to an $\epsilon$ value in the considered range.  To highlight the discrepancy among the three regimes, we color the sparse regime ($0<\epsilon<\frac{\beta}{2}$) in red, the moderately dense regime ($\frac{\beta}{2}\leq\epsilon\leq\frac{3\beta}{4}$) in green, and the strongly dense regime ($\frac{3\beta}{4}<\epsilon\leq\beta$) in blue.  We see that the three regimes have somewhat different behaviors for small positive values of $b$.  In particular, the sparse regime and the strongly dense regime experience two transitions (three different slopes), while the moderately dense regime experience three transitions (four different slopes).  Note that the difference in the number of transitions is restored at the intersection of the blue region and the red region.  Thus, the phase transition is in some sense ``continuous" across the regime boundaries --- the piecewise straight lines corresponding to $r(\beta, \epsilon, b)$'s exhibit smooth transition as $\epsilon$ increases from 0 to $\beta$.  The right panel of Figure~\ref{phase} provides a static view for each regime.  We plot $r(\beta, \epsilon, b)$ against $b$ for three values of $\epsilon$ corresponding to three different regimes: $\epsilon = 0.12$ (sparse regime), $\epsilon = 0.28$ (moderately dense regime), and $\epsilon = 0.4$ (strongly dense regime).  

\begin{figure}[!h]
\centering
\includegraphics[angle=0,width=6.5in]{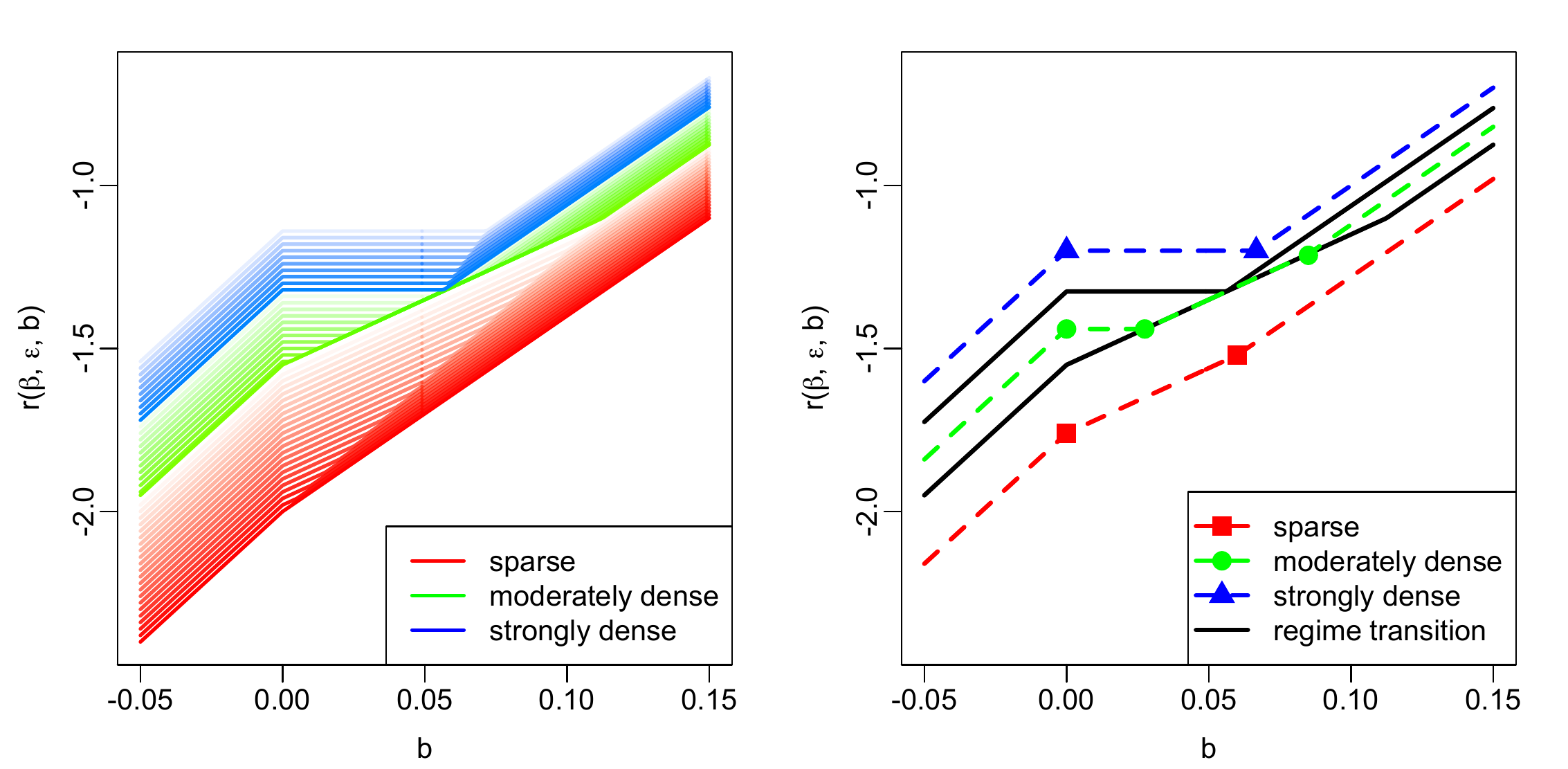}
\caption{Plot of the rate exponent $r(\beta, \epsilon, b)$ against the signal strength $b$.  In the sparse regime (\protect\includegraphics[height=0.4em]{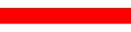}), $r(\beta, \epsilon, b)$ changes in the order $2\epsilon+8b-2, 2\epsilon+4b-2, \epsilon+6b-2$.  In the moderately dense regime (\protect\includegraphics[height=0.4em]{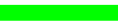}), $r(\beta, \epsilon, b)$ changes in the order $2\epsilon+8b-2, 2\epsilon-2, \beta+4b-2, \epsilon+6b-2$.  In the strongly dense regime (\protect\includegraphics[height=0.4em]{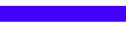}), $r(\beta, \epsilon, b)$ changes in the order $2\epsilon+8b-2, 2\epsilon-2, \epsilon+6b-2$.  Left panel: a continuum view of $r(\beta, \epsilon, b)$ as $\epsilon$ increases from 0 to $\beta = 0.45$ (color changes from red to blue).  Right panel: a static view of each regime: sparse ($\epsilon = 0.12$), moderately dense ($\epsilon = 0.28$), and strongly dense ($\epsilon = 0.4$).  Transition points are indicated by the knots on the dashed lines.}
\label{phase}
\end{figure}

Interestingly, in the two-sequence case, the regions $\{b: b\leq 0\}$ and $\{b: b>0\}$ appears to constitute the regions of weak signal and strong signal, respectively, regardless of the level of simultaneous sparsity.  This is in contrast to the one-sequence case where the dividing line is $b=0$ when $k_n\ll\sqrt{n}$, and $b=\frac{1-2\beta}{4}$ when $k_n\gg\sqrt{n}$.  We caution that this apparent ``reconciliation" in the two-sequence case is simply because the signal strengths are taken to be the same for both sequences $\mu$ and $\theta$ in the simplified results presented above.  

\begin{Remark}\label{commentab}
{\rm
When the signal strengths $r_n = n^a$ and $s_n = n^b$ of $\mu$ and $\theta$ are allowed to differ, it turns out that $\{(a, b): a\wedge b \leq 0\}$ characterizes the region of weak signal when $q_n\ll\sqrt{k_n}$, while $\{(a, b): a\vee b \leq 0\}\cup\{(a, b): a\wedge b \leq \frac{\beta-2\epsilon}{4}\}$ comprises the region of weak signal when $q_n\gg\sqrt{k_n}$.  We refer the readers to Appendix~\ref{generalEst} for more details.
}
\end{Remark}

\section{Detection}\label{detection}
There are strong connections between the problem of estimation and that of testing for quadratic functionals in the single Gaussian sequence model setting.  In this setting, the primary goal of the testing problem is to distinguish between $\theta = 0$ and $\theta \neq 0$.  Therefore, one can view the Gaussian sequence model as a signal plus noise model, where $\theta=0$ means that there is no signal, and testing whether $\theta$ is nonzero amounts to a signal detection problem.  It has been shown that test procedures based on estimators of the quadratic functional $Q(\theta)$ can be effective in detecting signals under various specifications of the parameter space (see, for example, \cite{CaiLow2005} and the references therein).

In this section, we explore the links between estimation and testing for quadratic functionals in the Gaussian two-sequence model.  In contrast to the one-sequence case, the main interest of testing in the two-sequence case is to distinguish between $\mu\star\theta = 0$ and $\mu\star\theta\neq 0$, where $\mu\star\theta = (\mu_1\theta_1, \ldots, \mu_n\theta_n)$ is the coordinate-wise product of $\mu$ and $\theta$.  Note that this is in effect a \emph{simultaneous} signal detection problem --- we are only interested in the case where both sequences contain signal.  As we shall see, the estimators $\Qhat_2$ in \eqref{estimator2} and $\Qhat_4$ in \eqref{estimator4} can be useful in the construction of test procedures.  

We will consider the hypothesis testing problem under an asymptotic minimax framework, where the size $n$ of the mean vector $\theta$ is the driving variable.  We first introduce some notation, which is applicable to both the one-sequence and two-sequence testing problem, i.e., think of $\theta$ below as a generic parameter.  Consider the testing problem
\[H_0: \theta\in\Theta_0(n), \qquad H_1: \theta\in\Theta_1(n),\]
where $\Theta_0(n)$ and $\Theta_1(n)$ are parameter spaces whose specification depends on $n$, and  $\Theta_0(n)\cap\Theta_1(n) = \emptyset$.  A test $\psi$ is a rule to accept or reject the null hypothesis based on the observed data.  Therefore, it is a measurable function of the observed data with values in $\{0, 1\}$.  The value $\psi = 1$ means that we reject $H_0$, and the value $\psi=0$ means that we do not reject $H_0$.  We measure the quality of a test $\psi$ by the sum of its maximal type I error (over $\Theta_0(n))$ and maximal type II error (over $\Theta_1(n))$:
\[S_n(\Theta_0(n), \Theta_1(n))(\psi) = \sup_{\theta\in\Theta_0(n)} E_\theta\psi + \sup_{\theta\in\Theta_1(n)}E_\theta(1-\psi).\]
We define the minimax total error probability for the hypothesis testing problem as the infimum of such total error probability over all tests:
\[S_n(\Theta_0(n), \Theta_1(n)) = \inf_\psi S_n(\Theta_0(n), \Theta_1(n))(\psi).\]
The goal is to establish the \emph{asymptotic detection boundary}, i.e., the conditions on $\Theta_0(n)$ and $\Theta_1(n)$ which separate the undetectable region (where $S_n(\Theta_0(n), \Theta_1(n))\longrightarrow 1$ as $n\longrightarrow\infty$) from the detectable region (where $S_n(\Theta_0(n), \Theta_1(n))\longrightarrow 0$ as $n\longrightarrow\infty$).  In the interior of undetectable region, signal is so weak that no tests can successfully separate $H_0$ from $H_1$: the sum of maximal type I and maximal type II error of any tests tends to one as $n$ tends to infinity.  On the contrary, in the interior of detectable region,  it is possible to find a test that has sum of maximal type I and maximal type II error tends to zero as $n$ diverges.   Along with the establishment of the asymptotic detection boundary, we want to find a test $\psi^*$ which can perfectly distinguish between $H_0$ and $H_1$ when $n$ is large, i.e., $S_n(\Theta_0(n), \Theta_1(n))(\psi^*)\longrightarrow 0$ as $n\longrightarrow\infty$, in the detectable region. 

We motivate our formulation of the simultaneous signal detection problem by the corresponding framework that has been established for the one-sequence signal detection problem \citep{Ingster1997, HallJin2010}:
\begin{equation}
H_0: \theta=0, \qquad H_1: \theta\in\Theta_1(\beta, b), \label{oneTestingA}
\end{equation}
where
\begin{equation}
\Theta_1(\beta, b) = \{\theta\in\R^n: \|\theta\|_0 = k_n, \theta\in\{0, \pm s_n\}^n\}, \label{oneTestingA2}
\end{equation}
$k_n = n^\beta$ with $0<\beta<1$ and $s_n = n^b$ with $b\in\R$.  In this formulation, there is no signal under the null hypothesis, whereas signal is constrained in terms of both sparsity and magnitude under the alternative hypothesis.  Intuitively, for a fixed $\beta$, the signal detection problem becomes easier when $b$ increases.  Similarly, for a fixed $b$, an increase in $\beta$ makes the detection problem easier.  In fact, the detection boundary under this framework is a mathematical formula describing the precise relationship between sparsity and signal strength.  It is a curve $b=\rho^*(\beta)$ that partitions $\{(\beta, b): 0<\beta<1, b\in\R\}$ into two regions: the detectable region where $b>\rho^*(\beta)$, and the undetectable region where $b<\rho^*(\beta)$.  

Similar to the problem of estimating the quadratic functional $Q(\theta)$ over the parameter space $\Theta(\beta, b)$ defined in \eqref{oneSpace}, the detection problem \eqref{oneTestingA} behaves differently over two regimes.  In the dense regime, $\half\leq\beta<1$, and the detection boundary is
\[\rho^*(\beta) = \frac{1-2\beta}{4}.\]
A simple test based on the quadratic functional $\Qhat_3$ defined in \eqref{estimator3} can be used here, by letting $\psi^* = \1(\Qhat_3\geq\lambda_n)$, where $\lambda_n\longrightarrow 0$ satisfies $\limsup_{n\longrightarrow\infty} \lambda_n n^{1-\beta-2b} < 1$.  With such choice of $\lambda_n$, $E_0\psi^* + \sup_{\theta\in\Theta_1(\beta, b)}E_\theta(1-\psi^*) \longrightarrow 0$ whenever $b>\rho^*(\beta)$.

In the sparse regime, $0<\beta<\half$. The detection boundary is
\[\rho^*(\beta) = 0.\]
Note that $\rho^*(\beta)$ is independent of $\beta$ in this case. The explanation here is that we are not using our microscope at the right resolution.  A more refined analysis reveals that if we parametrize $s_n$ at the order $s_n = \sigma\sqrt{d\log n}$, $d>0$, rather than at the algebraic order $s_n = n^b$, $b\in\R$, then signal is detectable whenever $d>\rho^*(\beta)$, where
\begin{equation}
\rho^*(\beta) = \left\{
\begin{array}{ll}
2(1-\sqrt{\beta})^2 &\text{if }0<\beta\leq\frac{1}{4}, \label{highercritBoundary}\\
1-2\beta &\text{if } \frac{1}{4}<\beta<\half. \\
\end{array}\right.
\end{equation}
The detection boundary \eqref{highercritBoundary} is the same as that given in \cite{DonohoJin2004}, where a mixture model is used.  A more general result extending to heteroscedatstic normal mixtures can be found in \cite{CaiJeng2011}.  The higher criticism test statistic is known to be optimally adaptive for the detection problem under the mixture model framework as considered in \cite{DonohoJin2004} and \cite{CaiJeng2011}.  Its optimal adaptivity is established in a regression context in \cite{AriasCastro2011}, where the single Gaussian sequence signal detection \eqref{oneTestingA} serves as a special case, whereas the max-type test statistic is only optimal over the region $0<\beta\leq\frac{1}{4}$ (see, e.g., Theorem 1 in \cite{AriasCastro2011}). 

In the interest of connecting testing problem with estimation of quadratic functional, one might wonder if some form of sum-of-squares type test statistic such as that used in the dense regime can be effective in the sparse regime.  A natural candidate to consider is the estimator $\Qhat_1$ of $Q(\theta)$ defined in \eqref{estimator1}.  Indeed, the test $\psi^* = \1(\Qhat_1\geq\lambda_n)$, where $\lambda_n = \frac{\log n}{n}$ can asymptotically distinguish between $H_0$ and $H_1$ if $b>0$, or if $s_n = \sigma\sqrt{d\log n}$ with $d$ sufficiently large.  A rough analysis shows that the test procedure works for $d>4$.  Since we only have an upper bound for the mean and variance of $\Qhat_1$ (and not the exact values), it will be challenging if not impossible to derive the smallest possible value of $d$ where $\psi^*$ can be effective.

To generalize detection problem of the form \eqref{oneTestingA}, \eqref{oneTestingA2} to the two-sequence case, consider the following parameter spaces:
\begin{align}
\Omega_0(\beta, a, b) &= \{(\mu, \theta): \|\mu\|_0\leq k_n, \|\theta\|_0\leq k_n, \mu\in\{0, \pm r_n\}^n, \theta\in\{0, \pm s_n\}^n, \nonumber\\
&\qquad\qquad\quad\|\mu\star\theta\|_0 = 0\}, \nonumber\\
\Omega_1(\beta, \epsilon, a, b) &= \{(\mu, \theta): \|\mu\|_0\leq k_n, \|\theta\|_0\leq k_n, \mu\in\{0, \pm r_n\}^n, \theta\in\{0, \pm s_n\}^n, \nonumber\\
&\qquad\qquad\quad\|\mu\star\theta\|_0 = q_n\}, \label{twoTestingA2}
\end{align}
where $k_n = n^\beta, q_n = n^\epsilon$ with $0<\epsilon\leq\beta<\half$, and $r_n = n^a, s_n = n^b$ with $a, b\in\R$.  Suppose that we want to test
\begin{equation}
H_0: (\mu, \theta)\in\Omega_0(\beta, a, b), \qquad H_1: (\mu, \theta)\in\Omega_1(\beta, \epsilon, a, b). \label{twoTestingA}
\end{equation}
Essentially, we are testing whether $q_n = 0$ on condition that $(\mu, \theta)$ satisfies the required sparsity and signal strength constraints.  It is, perhaps, unsurprising that the testing problem being characterized by two regimes: the sparse regime where $0<\epsilon<\frac{\beta}{2}$, and the dense regime $\frac{\beta}{2}\leq\epsilon\leq\beta$.  As the testing problem now involves four parameters: $a, b, \beta$, and $\epsilon$, it is easier to describe the detectable versus undetectable regions, instead of the detection boundary.

\begin{Theorem}\label{two-det-Q}
We state the detectable and undetectable regions in two cases:
\begin{enumerate}[(a)]
\item In the sparse regime, $0<\epsilon<\frac{\beta}{2}$.  The undetectable region is $\{(a, b): a\wedge b\leq 0\}$, and the detectable region is $\{(a, b): a\wedge b> 0\}$.  The test $\psi^* = \1(\Qhat_2\geq\lambda_n)$, where $\Qhat_2$ is as defined in \eqref{estimator2} with $\tau_n = \log n$, $\lambda_n = \half n^{\epsilon+2a+2b-1}$, asymptotically separates $H_0$ from $H_1$ over the detectable region.
\item In the dense regime, $\frac{\beta}{2}\leq\epsilon\leq\beta$.  The undetectable region is $\{(a, b): a\wedge b<\frac{\beta-2\epsilon}{4}\text{ or }a\vee b\leq 0\}$, and the detectable region is $\{(a, b): a\wedge b> \frac{\beta-2\epsilon}{4}\text{ and }a\vee b>0\}$.  The test $\psi^* = \1(\Qhat_4\geq\lambda_n)$, where $\Qhat_4$ is as defined in \eqref{estimator4} with $\tau_n = 4\log n$, $\lambda_n = \half n^{\epsilon+2a+2b-1}$, asymptotically separates $H_0$ from $H_1$ over the detectable region.
\end{enumerate}
\end{Theorem}

\begin{figure}[!h]
\centering
\includegraphics[angle=0,width=6.5in]{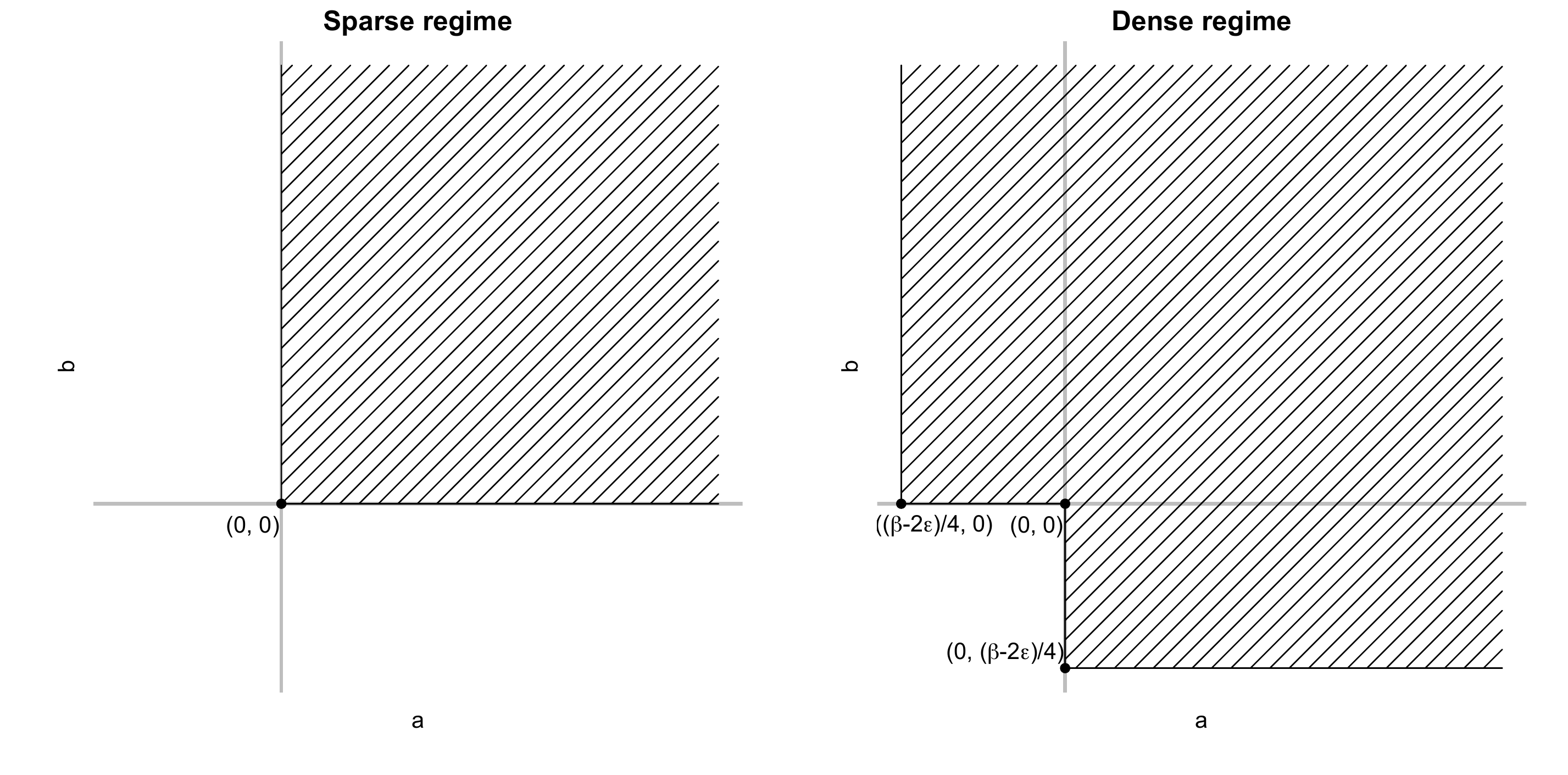}
\caption{Plot of detectable versus undetectable regions for the sparse regime (left) and the dense regime (right).  The shaded area corresponds to the detectable region, while the unshaded area corresponds to the undetectable region.}
\label{detectable}
\end{figure}

In Figure~\ref{detectable}, we plot the detectable and the undetectable regions for both sparse and dense regimes.  The detectable region in the dense regime is larger than that in the sparse regime.  Interestingly, in the dense regime, signal is detectable provided that the signal strength of at least one of the sequences is large enough (and the signal strength of the other sequence is not too weak).  In contrast, in the sparse regime, signal is only detectable when both sequences admit sufficiently strong signals.

Note that the undetectable region for the detection problem \eqref{twoTestingA} is essentially the same as the region where $\Qhat_0 = 0$ attains the optimal rate of estimation for $Q(\mu, \theta)$ over the parameter space $\Omega(\beta, \epsilon, a, b)$ defined in \eqref{paramTwoab} of Appendix~\ref{generalEst}.  To see the connection, compare Figure \ref{detectable} with Remark~\ref{commentab} given at the end of Section~\ref{phaseplot}, and the results in Appendix~\ref{generalEst}.  

One may notice that the resolution of our microscope is not high enough in some sense: the condition $a\wedge b = 0$ when $0<\epsilon<\frac{\beta}{2}$ and the condition $a\vee b = 0$ when $\frac{\beta}{2}\leq\epsilon\leq\beta$ does not involve $\epsilon$ and $\beta$.  A more refined analysis may involve the parameterization $r_n\wedge s_n = \sigma\sqrt{c\log n}$ or $r_n\vee s_n = \sigma\sqrt{d\log n}$ for some $c, d>0$.  One can show that in such cases, the test procedures based on $\Qhat_2$ and $\Qhat_4$ are still effective for $c, d$ sufficiently large.  A detailed analysis of the exact detection boundary along this order is beyond the scope of the paper.

\section{Simulation}\label{simulation}
In this section, we perform some simulation studies to compare the performance of the three estimators $\Qhat_0 = 0, \Qhat_2$ as in \eqref{estimator2}, and $\Qhat_4$ as in \eqref{estimator4}, under different scenarios.  We compute the mean squared error (MSE) of the three estimators and show that our simulation results is compatible with the theoretical results given in Section~\ref{twosample}.

So far, we have assumed that the noise level $\sigma$ is known. In practice, $\sigma$ is typically unknown and needs to be estimated. Under the sparse setting of the present paper, $\sigma$ is easily estimable. Denote by $M\in \mathbb{R}^{2n}$ with $M_{2i-1}=X_i$ and $M_{2i}=Y_i$ for $i=1, ..., n$. A simple robust estimator of the noise level $\sigma$ is the following median absolute deviation (MAD) estimator:
\[
\hat \sigma = \frac{{\rm median} |M_j - {\rm median}(M_j) |}{0.6745}.
\]

We consider simulation studies over a range of sample size $n$, sparsity $k_n = n^\beta$, simultaneous sparsity $q_n = n^\epsilon$, and signal strength $s_n = n^b$.  More specifically, we take $n \in \{10^3, 10^4, \ldots, 10^7\}$, $\beta = 0.45$ for individual sequences, $b \in \{-0.1, 0.15, 0.2\}$, and three values of simultaneous sparsity, one for each regime: $\epsilon = 0.02$ (sparse regime), $\epsilon = 0.3$ (moderately dense regime) and $\epsilon = 0.44$ (strongly dense regime).  Figure~\ref{Sim} is the plot of the MSE (averaged over 500 replications) of the three estimators over different sample sizes (in the log-log scale), for each combination of simultaneous sparsity and signal strength.  

\begin{figure}[!h]
\centering
\includegraphics[angle=0,width=6.5in]{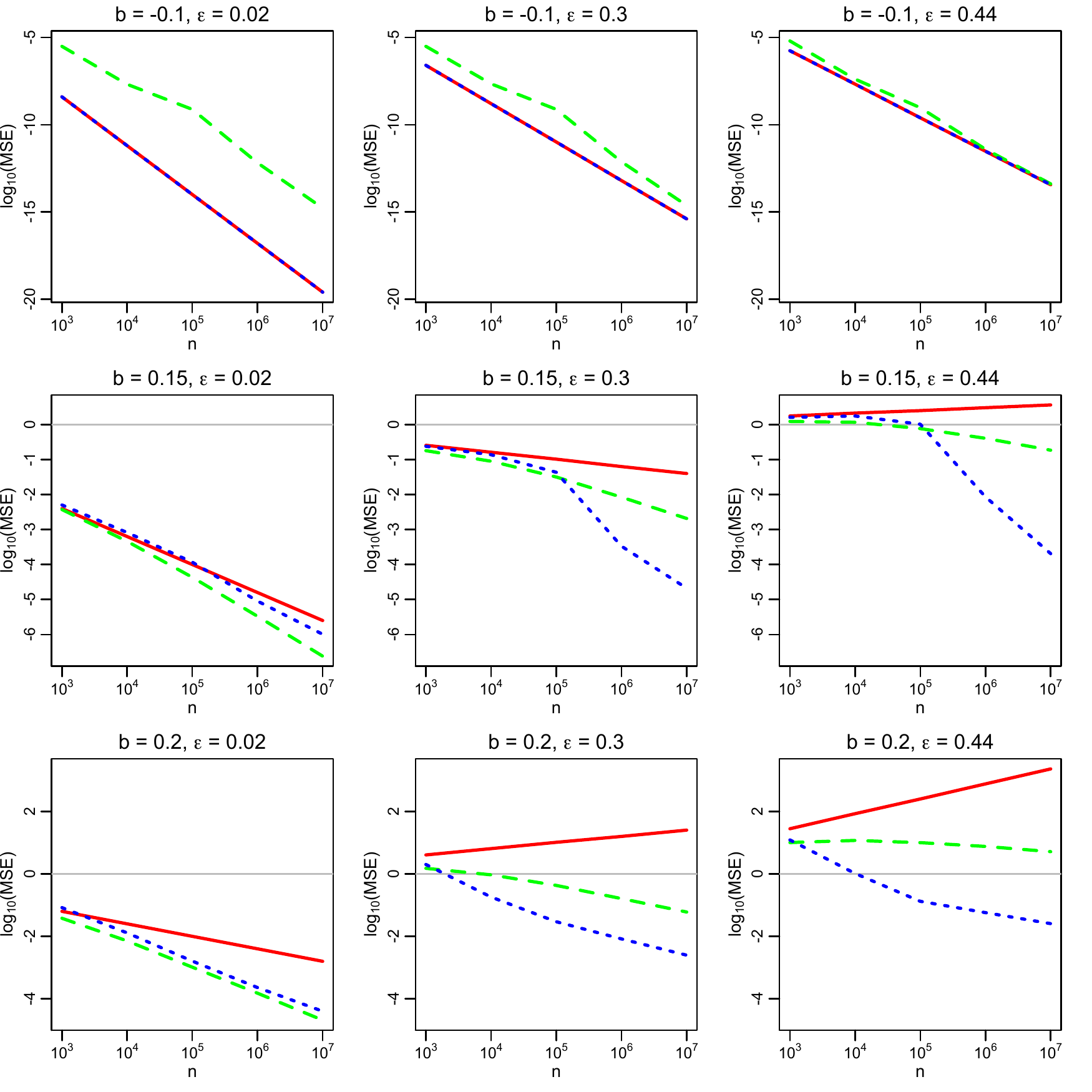}
\caption{Plot of MSE for the estimators $\Qhat_0$ (\protect\includegraphics[height=0.4em]{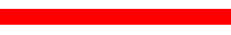}), $\Qhat_2$ (\protect\includegraphics[height=0.4em]{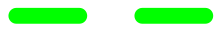}), and $\Qhat_4$ (\protect\includegraphics[height=0.4em]{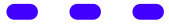}) over different sample sizes $n \in \{10^3, \ldots, 10^7\}$, in the log-log scale.  The columns are ordered from left to right as $\epsilon = 0.02$ (sparse regime), $\epsilon = 0.3$ (moderately dense regime), and $\epsilon = 0.44$ (strongly dense regime).  The rows are ordered from top to bottom in increasing signal strength: $b \in \{-0.1, 0.15, 0.2\}$.  Horizontal grey line corresponds to $\text{MSE} = 1$, and it serves to distinguish between $\text{MSE}\longrightarrow 0$ (negative slope) and $\text{MSE}\longrightarrow\infty$ (positive slope).}
\label{Sim}
\end{figure}

The theoretical results in Section~\ref{twosample} indicate that for $\Qhat = \Qhat_0, \Qhat_2$, or $\Qhat_4$,
\[\sup_{(\mu, \theta)\in\Omega(\beta, \epsilon, b)} E(\Qhat-Q(\mu, \theta))^2 \asymp n^{r(\beta, \epsilon, b)}\]
for some rate exponent $r(\beta, \epsilon, b)$ (modulo a logarithmic factor when applicable).  Thus, it is not surprising that the results in Figure~\ref{Sim} (mostly) exhibit a linear pattern.  When the signal is weak with $b =-0.1$ (see the first row of Figure~\ref{Sim}), we see that $\Qhat_0$ (solid red line) and $\Qhat_4$ (dotted blue line) have the lowest mean square error.  Note that we expect $\Qhat_0$ to be optimal when the signal is weak.  We observe that $\Qhat_4$ is nearly as good as $\Qhat_0$ from Figure~\ref{Sim}.  This is because when the signal is weak, the thresholding step $\1(X_i^2\vee Y_i^2\geq\sigma^2\tau_n)$ thresholds both noise and weak signals, and the de-bias term $\eta$ is extremely small when $n$ is moderately large, resulting in $\Qhat_4 \approx \Qhat_0 = 0$.  As signal becomes sufficiently strong ($b \in \{0.15, 0.2\}$), $\Qhat_2$ starts to dominate in the sparse regime ($\epsilon = 0.02$) while $\Qhat_4$ dominates in the moderately dense and strongly dense regimes ($\epsilon \in \{0.3, 0.44\}$).  When the signal is sufficiently large ($b \in \{0.15, 0.2\}$), $\Qhat_0$ is clearly suboptimal.  In particular, in the case where signal is both dense and strong ($b=0.2, \epsilon\in\{0.3, 0.44\}$), the MSE of $\Qhat_0$ diverges to infinity, as indicated by the positive slope of the solid red line.  Note also that as either $\epsilon$  or $b$ increases, MSE increases, as can be seen by the flattening or reversing of slopes towards the right end or bottom of the plot panel.  This is compatible with the fact that $r(\beta, \epsilon, b)$ increases with respect to both $\epsilon$ and $b$.

\section{Discussion}\label{discussion}
In this paper, we discuss the estimation of the quadratic functional $Q(\mu, \theta) = \frac{1}{n}\sum\mu_i^2\theta_i^2$ over a family of parameter spaces where the magnitude and sparsity of both $\mu$ and $\theta$ are constrained.  We show that the minimax rates of convergence display interesting phase transitions over three distinct regimes: the sparse regime, the moderately dense regime, and the strongly dense regime.  We also demonstrate an application of the estimators of the quadratic functional to a closely related simultaneous signal detection problem, and show that the resulting test procedures are effective in detecting simultaneous signals over the detectable region.  Throughout our analysis, we highlight distinctions and similarities between the one-sequence and two-sequence problems, for both estimation and testing.  

It will be interesting to generalize our study of the two-sequence estimation problems in numerous aspects.  In Appendix~\ref{general}, we show that the optimal rates of estimation for $Q(\mu, \theta)$ continue to subsume the aforementioned three regimes, when $\mu$ and $\theta$ are allowed unequal signal strengths.  Nonetheless, the distinction between the sparse and dense regime is more apparent in this setting.  In the sparse regime, estimation is only desirable when the signal strength of both sequences are sufficiently strong.  In contrast, in the dense regime, estimation is desirable whenever at least one sequence admits sufficiently strong signal (and the signal strength of the other sequence is not too weak).  We also examine in Appendix~\ref{general} the behavior of the estimation problem when signal strength is incorporated through $\ell_2$-norm rather than $\ell_\infty$-norm.  Unlike the one-sequence estimation problem, in the two-sequence case the minimax rates of convergence are to some extent degenerate under the $\ell_2$-norm constraint.  Thus, it is reasonable to suspect that the one-sequence and two-sequence estimation problems are not that resembling after all.  A more refined analysis of the characteristics of both one-sequence and two-sequence problems demands an examination of their respective behaviors under the $\ell_p$-norm constraint on the signal strength, for $p\in (0, \infty]$, and is beyond the scope of the paper.

\section{Proofs of Theorem~\ref{two-inf-sparse-up} and Theorem~\ref{two-inf-sparse}}\label{proofs}
In this section, we present the proofs of Theorem~\ref{two-inf-sparse-up} and Theorem~\ref{two-inf-sparse}, which concern estimation results of $Q(\mu, \theta)$ in the sparse regime.  For reason of space, we relegate the proofs of other main results in the paper to Appendix~\ref{prooftwosample}.  

Henceforth, we omit the subscripts $n$ in $k_n, q_n, s_n$ and $\tau_n$ that signifies their dependence on the sample size.  We denote by $\psi_\mu$ the density of a Gaussian distribution with mean $\mu$ and variance $\sigma^2$, and we denote by $\ell(n, k)$ the class of all subsets of $\{1, \ldots, n\}$ of $k$ distinct elements.  Finally, $c$ and $C$ denote constants that may vary for each occurrence.

\subsection{Proof of Theorem~\ref{two-inf-sparse-up}}\label{sparse-upper-proof}
The proof of Theorem~\ref{two-inf-sparse-up} involves a careful analysis of the bias and variance of the estimator $\Qhat_2$ (defined in \eqref{estimator2}).  We need the following lemma from \cite{CaiLow2005} (Lemma 1, page 2939) for proving Theorem~\ref{two-inf-sparse-up}.
\begin{Lemma}\label{qbound2}
Let $Y\sim N(\theta, \sigma^2)$ and let $\theta_0 = E_0(Y^2-\sigma^2\tau)_+$, where the expectation is taken under $\theta=0$.  Then for $\tau\geq 1$ and $\widehat{\theta^2} = (Y^2-\sigma^2\tau)_+-\theta_0$,
\begin{align*}
|\theta_0| &\leq \frac{4\sigma^2}{\sqrt{2\pi}\tau^{1/2} e^{\tau/2}}, \\
|E(\widehat{\theta^2}) -\theta^2| &\leq \min\{2\sigma^2\tau, \theta^2\}, \\
\Var(\widehat{\theta^2}) &\leq 6\sigma^2\theta^2 + \sigma^4\frac{4\tau^{1/2}+18}{e^{\tau/2}}.
\end{align*}
\end{Lemma}

Lemma~\ref{varboundlemma} is an immediate consequence of Lemma~\ref{qbound2}.
\begin{Lemma}\label{varboundlemma}
Let $Y\sim N(\theta, \sigma^2)$ and let $\theta_0 = E_0(Y^2-\sigma^2\tau)_+$, where the expectation is taken under $\theta=0$.  Then for $\tau\geq 1$,
\[(E(Y^2-\sigma^2\tau)_+ - \theta_0)^2 \leq \max\bigg\{6\sigma^2\theta^2 + \sigma^4\frac{4\tau^{1/2}+18}{e^{\tau/2}}, 10\theta^4\bigg\}.\]
\end{Lemma}
To streamline the presentation of proof of Theorem~\ref{two-inf-sparse-up}, we defer the proof of Lemma~\ref{varboundlemma} to the end of Appendix~\ref{two-upper}.

\begin{proof}[Proof of Theorem~\ref{two-inf-sparse-up}]
We first bound the bias of the estimator $\Qhat_2$ defined in \eqref{estimator2}.  Using the equality
\[AB-ab = (A-a)(B-b) + a(B-b) + b(A-a),\]
the independence of $X_i$ and $Y_i$, and the triangle inequality, we get
\begin{align*}
& \Big|E_{(\mu_i, \theta_i)}\{[(X_i^2-\sigma^2\tau)_+ - \mu_0][(Y_i^2-\sigma^2\tau)_+ - \theta_0]\} - \mu_i^2\theta_i^2\Big| \\
&\leq \Big|E_{\mu_i}[(X_i^2-\sigma^2\tau)_+ - \mu_0] - \mu_i^2\Big|\cdot \Big|E_{\theta_i}[(Y_i^2-\sigma^2\tau)_+ - \theta_0] - \theta_i^2\Big| \\
&\qquad + \mu_i^2\Big|E_{\theta_i}[(Y_i^2-\sigma^2\tau)_+ - \theta_0] - \theta_i^2\Big| + \theta_i^2\Big|E_{\mu_i}[(X_i^2-\sigma^2\tau)_+ - \mu_0]-\mu_i^2\Big| \\ 
&\leq \min\{2\sigma^2\tau, \mu_i^2\}\min\{2\sigma^2\tau, \theta_i^2\} + \mu_i^2\min\{2\sigma^2\tau, \theta_i^2\} + \theta_i^2\min\{2\sigma^2\tau, \mu_i^2\} \\
&\leq 2\mu_i^2\min\{2\sigma^2\tau, \theta_i^2\} + 2\theta_i^2\min\{2\sigma^2\tau, \mu_i^2\},
\end{align*}
the second inequality follows from Lemma~\ref{qbound2}.  It follows that for $(\mu, \theta)\in\Omega(\beta, \epsilon, b)$ and $\tau\geq 1$,  
\begin{align*}
&|E_{(\mu, \theta)}(\Qhat_2) - Q(\mu, \theta)| \\
&= \bigg|\frac{1}{n}\sum_{i=1}^nE_{(\mu_i, \theta_i)}\{[(X_i^2-\sigma^2\tau)_+ - \mu_0][(Y_i^2-\sigma^2\tau)_+ - \theta_0]\} - \frac{1}{n}\sum_{i=1}^n\mu_i^2\theta_i^2\bigg| \\
&\leq \frac{2}{n}\sum_{i=1}^n \big[\mu_i^2\min\{2\sigma^2\tau, \theta_i^2\} + \theta_i^2\min\{2\sigma^2\tau, \mu_i^2\}\big] \\
&\leq \frac{4}{n}\min\{2\sigma^2qs^2\tau, qs^4\},
\end{align*}
the second inequality follows from the fact that for $(\mu, \theta)\in\Omega(\beta, \epsilon, b)$, there are at most $q$ entries that are simultaneously nonzero for $\mu$ and $\theta$.  

We now proceed to bounding the variance of $\Qhat_2$.  Applying the equality
\[\Var(AB) = \Var(A)\Var(B) + [E(A)]^2\Var(B) + [E(B)]^2\Var(A),\]
for $\tau\geq 1$, we have
\begin{align*}
&\Var_{(\mu_i, \theta_i)}\{[(X_i^2-\sigma^2\tau)_+ - \mu_0][(Y_i^2-\sigma^2\tau)_+ - \theta_0]\} \\
&= \Var_{\mu_i}[(X_i^2-\sigma^2\tau)_+ - \mu_0]\Var_{\theta_i}[(Y_i^2-\sigma^2\tau)_+ - \theta_0] \\
&\qquad + [E_{\mu_i}(X_i^2-\sigma^2\tau)_+ - \mu_0]^2\Var_{\theta_i}[(Y_i^2-\sigma^2\tau)_+ - \theta_0] \\
&\qquad + [E_{\theta_i}(Y_i^2-\sigma^2\tau)_+ - \theta_0]^2\Var_{\mu_i}[(X_i^2-\sigma^2\tau)_+ - \mu_0] \\
&\leq 3 \bigg[6\sigma^2\mu_i^2 + \sigma^4\frac{4\tau^{1/2}+18}{e^{\tau/2}}\bigg]\bigg[6\sigma^2\theta_i^2 + \sigma^4\frac{4\tau^{1/2}+18}{e^{\tau/2}}\bigg] \\
&\qquad + 10\mu_i^4\bigg[6\sigma^2\theta_i^2 + \sigma^4\frac{4\tau^{1/2}+18}{e^{\tau/2}}\bigg] + 10\theta_i^4\bigg[6\sigma^2\mu_i^2 + \sigma^4\frac{4\tau^{1/2}+18}{e^{\tau/2}}\bigg], 
\end{align*}
the inequality follows from Lemma~\ref{qbound2} and Lemma~\ref{varboundlemma}.  Thus, for $(\mu, \theta)\in\Omega(\beta, \epsilon, b)$ and $\tau\geq 1$,
\begin{align*}
&\Var_{\mu, \theta}(\Qhat_2) \\
&= \frac{1}{n^2}\sum_{i=1}^n \Var_{(\mu_i, \theta_i)}\{[(X_i^2-\sigma^2\tau)_+ - \mu_0][(Y_i^2-\sigma^2\tau)_+ - \theta_0]\} \\
&\leq \frac{3}{n^2}\sum_{i=1}^n \bigg[6\sigma^2\mu_i^2 + \sigma^4\frac{4\tau^{1/2}+18}{e^{\tau/2}}\bigg]\bigg[6\sigma^2\theta_i^2 + \sigma^4\frac{4\tau^{1/2}+18}{e^{\tau/2}}\bigg] \\
&\qquad + \frac{10}{n^2}\sum_{i=1}^n\mu_i^4\bigg[6\sigma^2\theta_i^2 + \sigma^4\frac{4\tau^{1/2}+18}{e^{\tau/2}}\bigg] + \frac{10}{n^2}\sum_{i=1}^n\theta_i^4\bigg[6\sigma^2\mu_i^2 + \sigma^4\frac{4\tau^{1/2}+18}{e^{\tau/2}}\bigg] \\
&\leq \frac{3}{n^2}\bigg[36\sigma^4qs^4 + 12\sigma^6ks^2\bigg(\frac{4\tau^{1/2}+18}{e^{\tau/2}}\bigg) + n\sigma^8\bigg(\frac{4\tau^{1/2}+18}{e^{\tau/2}}\bigg)^2\bigg] \\
&\qquad+ \frac{20}{n^2}\bigg[6\sigma^2qs^6 + \sigma^4ks^4\bigg(\frac{4\tau^{1/2}+18}{e^{\tau/2}}\bigg)\bigg].
\end{align*}
Combining the bias and variance term, we get, for $\tau\geq 1$,
\begin{align*}
&\sup_{(\mu, \theta)\in \Omega(\beta, \epsilon, b)}E_{(\mu, \theta)}(\Qhat_2-Q(\mu, \theta))^2 \\
&\leq \frac{C}{n^2}\bigg[\min\{q^2s^4\tau^2, q^2s^8\} +\max\bigg\{qs^4, qs^6, ks^2\bigg(\frac{4\tau^{1/2}+18}{e^{\tau/2}}\bigg), \\
&\qquad\qquad\qquad\qquad\qquad\qquad\qquad\qquad\qquad ks^4\bigg(\frac{4\tau^{1/2}+18}{e^{\tau/2}}\bigg), n\bigg(\frac{4\tau^{1/2}+18}{e^{\tau/2}}\bigg)^2\bigg\}\bigg] \\
&= \frac{C}{n^2}\bigg[\min\{n^{2\epsilon+4b}\tau^2, n^{2\epsilon+8b}\} + \max\bigg\{n^{\epsilon+4b}, n^{\epsilon+6b}, n^{\beta+2b}\bigg(\frac{4\tau^{1/2}+18}{e^{\tau/2}}\bigg), \\
&\qquad\qquad\qquad\qquad\qquad\qquad\qquad\qquad\qquad n^{\beta+4b}\bigg(\frac{4\tau^{1/2}+18}{e^{\tau/2}}\bigg), n\bigg(\frac{4\tau^{1/2}+18}{e^{\tau/2}}\bigg)^2\bigg\}\bigg]. 
\end{align*}
Suppose that $b>0$. Then letting $\tau = \log n$ leads to
\[\sup_{(\mu, \theta)\in \Omega(\beta, \epsilon, b)}E_{(\mu, \theta)}(\Qhat_2-Q(\mu, \theta))^2 \leq C\bigg[n^{2\epsilon+4b-2}(\log n)^2 + n^{\epsilon+6b-2}\bigg] \leq C\gamma_n(\beta, \epsilon, b),\]
where
\begin{equation}
\gamma_n(\beta, \epsilon, b) = \left\{
\begin{array}{ll}
n^{2\epsilon+4b-2}(\log n)^2 &\text{if }0<b\leq\frac{\epsilon}{2}, \\
n^{\epsilon+6b-2} &\text{if }b>\frac{\epsilon}{2}.
\end{array}\right. \label{Q3rates}
\end{equation}
From the calculation above, one can also check that when $0<\epsilon\leq\beta$ and $s = \sigma\sqrt{d\log n}$, 
\[\sup_{\substack{(\mu, \theta): \|\mu\|_0\leq k_n, \|\theta\|_0\leq k_n, \\\|\mu\|_\infty\leq \sigma\sqrt{d\log n}, \|\theta\|_\infty\leq \sigma\sqrt{d\log n}}}E_{(\mu, \theta)}(\Qhat_2-Q(\mu, \theta))^2 \leq Cn^{2\epsilon-2}(\log n)^4.\]
\end{proof}

\subsection{Proof of Theorem~\ref{two-inf-sparse}}\label{sparse-lower-proof}
To prove Theorem~\ref{two-inf-sparse}, it suffices to show that for $0<\beta<\half$,
\begin{equation*}
\gamma_n(\beta, \epsilon, b) \geq \left\{
\begin{array}{llll}
n^{2\epsilon+4b-2}(\log n)^2 &\text{if }b>0, &\text{for }0<\epsilon<\frac{\beta}{2}, & \qquad\qquad(\text{Case } 1) \\
n^{2\epsilon+8b-2} &\text{if }b\leq 0, &\text{for }0<\epsilon\leq\beta, & \qquad\qquad(\text{Case } 2) \\
n^{\epsilon+6b-2} &\text{if }b>0, &\text{for } 0<\epsilon\leq\beta. & \qquad\qquad(\text{Case } 3) \\
\end{array}\right.
\end{equation*}
For individual regions in $\{(\beta, \epsilon, b): 0<\epsilon<\frac{\beta}{2}, 0<\beta<\half, b\in\R\}$, the minimax rate of convergence is then given by the sharpest rate among all cases in which the region belongs to.  For instance, the region $\{(\beta, \epsilon, b): 0<\epsilon<\frac{\beta}{2}, 0<\beta<\half, b>\frac{\epsilon}{2}\}$ is included in Case 1 and Case 3, hence $\gamma_n(\beta, \epsilon, b) \geq \max\{n^{2\epsilon+4b-2}(\log n)^2, n^{\epsilon+6b-2}\} = n^{\epsilon+6b-2}$.  

To establish the desired lower bounds, for each cases, we construct two priors $f$ and $g$ that have small chi-square distance but a large difference in the expected values of the resulting quadratic functionals, and then applying the Constrained Risk Inequality (CRI) in \cite{BrownLow1996}.  The choice of priors $f$ and $g$ are crucial in deriving sharp lower bound for the estimation problem. In fact, the fundamental difference between different phases in the sparse regime for the estimation of $Q(\mu, \theta)$ can be seen from the choice $f$ and $g$.  For some background on lower bound technique, see Appendix~\ref{tools}.

\begin{proof}[Proof of Case 1]
Our proof builds on arguments similar to that used in \cite{CaiLow2004} and \cite{Baraud2002}, who considered the one-sequence estimation problem.  We first follow the lines of the proof of Theorem 7 in \cite{CaiLow2004}, and then apply a result from \cite{Aldous1983} as was done in \cite{Baraud2002}.  Let 
\[f(x_1, \ldots, x_n, y_1, \ldots, y_n) = \prod_{i=1}^k \psi_s(x_i)\prod_{i=k+1}^n\psi_0(x_i)\prod_{i=1}^n \psi_0(y_i).\]
For $I\in\ell(k, q)$, let
\[g_I(x_1, \ldots, x_n, y_1, \ldots, y_n) = \prod_{i=1}^k \psi_s(x_i)\prod_{i=k+1}^n\psi_0(x_i)\prod_{i=1}^k\psi_{\theta_i}(y_i)\prod_{i=k+1}^n \psi_0(y_i),\]
where $\theta_i = \rho\1(i\in I)$ with $\rho>0$, and let
\[g = \frac{1}{\binom{k}{q}}\sum_{I\in\ell(k, q)} g_I.\]
In both $f$ and $g$, the sequence $\mu = (s, \ldots, s, 0, \ldots, 0)$ is taken to be the same. However, $\theta$ is taken to be all zeros in $f$ but is taken as a mixture in $g$.  The nonzero coordinates of $\theta$ is mixed uniformly over the support of $\mu$ at a common magnitude $\rho$, whose value is yet to be determined.  Our choice of $f$ and $g$ essentially reduces the two-sequence problem to the case where we only have one Gaussian mean sequence of length $k$ with $q$ nonzero coordinates, hence explains the correspondence between the sparse regime in the two-sequence case ($q<\sqrt{k}$) and the sparse regime in the one-sequence case ($k<\sqrt{n}$).  

We now compute the chi-square affinity between $f$ and $g$, which bears the expression
\begin{equation}
\int\frac{g^2}{f} = \frac{1}{\binom{k}{q}^2}\sum_{I\in\ell(k, q)}\sum_{J\in\ell(k, q)}\int\frac{g_Ig_J}{f}. \label{mixprior}
\end{equation}
For $I, J\in\ell(k, q)$, let $m = \text{Card}(I\cap J)$.  Then
\begin{align*}
\int\frac{g_Ig_J}{f}  &= \prod_{i=1}^k\int\frac{\psi_{\rho\1(i\in I)}(y_i)\cdot\psi_{\rho\1(i\in J)}(y_i)}{\psi_0(y_i)} \ dy_i \\
&= \bigg[\int\psi_0(y)\ dy\bigg]^{k-2q+m}\bigg[\int\psi_{\rho}(y)\ dy\bigg]^{2q-2m}\bigg[\int\frac{\psi_{\rho}^2(y)}{\psi_0(y)}\ dy\bigg]^m \\
&= \exp\bigg(\frac{m\rho^2}{\sigma^2}\bigg).
\end{align*}
It follows that
\[\int\frac{g^2}{f} = E\bigg[\exp\bigg(\frac{M\rho^2}{\sigma^2}\bigg)\bigg],\]
where $M$ has a hypergeometric distribution
\begin{equation}
P(M = m) = \frac{\binom{q}{m}\binom{k-q}{q-m}}{\binom{k}{q}}. \label{hyperprob}
\end{equation}
As shown in \cite{Aldous1983}, $M$ has the same distribution as the conditional expectation $E(\tilde{M}|\mathcal{B})$, where $\tilde{M}$ is a Binomial($q$, $\frac{q}{k}$) random variable and $\mathcal{B}$ is a suitable $\sigma$-algebra.  Coupled with Jensen's inequality, this implies that
\[\int\frac{g^2}{f} \leq E\bigg[\exp\bigg(\frac{\tilde{M}\rho^2}{\sigma^2}\bigg)\bigg] = \bigg(1-\frac{q}{k}+\frac{q}{k}e^{\rho^2/\sigma^2}\bigg)^q.\]
Taking $\rho = \sigma\sqrt{(\beta-2\epsilon)\log n}$ gives
\[e^{\rho^2/\sigma^2} = n^{\beta-2\epsilon} = \frac{k}{q^2},\]
hence
\[\int\frac{g^2}{f} \leq \bigg(1 + \frac{1}{q}\bigg)^q \leq e.\]
Since $Q(\mu, \theta) = 0$ under $f$ and $Q(\mu, \theta) = \frac{1}{n}qs^2\rho^2$ under $g$, it follows from CRI that 
\[R^*(n, \Omega(\beta, \epsilon, b)) \geq c\bigg(\frac{1}{n}qs^2\rho^2\bigg)^2 = cn^{2\epsilon+4b-2}(\log n)^2.\]
\end{proof}

\begin{proof}[Proof of Case 2]
Let
\[f(x_1, \ldots, x_n, y_1, \ldots, y_n) = \prod_{i=1}^n \psi_0(x_i)\prod_{i=1}^n \psi_0(y_i)\]
For $I\in\ell(n, q)$, let
\[g_I(x_1, \ldots, x_n, y_1, \ldots, y_n) = \prod_{i=1}^n\psi_{\mu_i}(x_i)\prod_{i=1}^n\psi_{\theta_i}(y_i),\]
where $\mu_i = \theta_i = \rho\1(i\in I)$ with $\rho>0$, and let
\[g = \frac{1}{\binom{n}{q}}\sum_{I\in\ell(n, q)} g_I.\]
Contrast the choice of $f$ an $g$ here with that used in the proof of Case 1.  Rather than fixing $\mu$ and mixing nonzero coordinates of $\theta$ over the support of $\mu$, in this case mixing is done over all $n$ positions using nonzero coordinates of $\mu$ and $\theta$ simultaneously.

Similar calculation as that used in the proof of Case 1 yields
\begin{equation}
\int\frac{g^2}{f} \leq \bigg(1 - \frac{q}{n} + \frac{q}{n}e^{2\rho^2/\sigma^2}\bigg)^q. \label{letrho}
\end{equation}
Now take $\rho = s = n^b$.  Since $b<0$, it follows that when $n$ is sufficiently large, 
\[e^{2\rho^2/\sigma^2} \leq n^{1-2\epsilon} = \frac{n}{q^2},\]
hence
\[\int\frac{g^2}{f} \leq \bigg(1+\frac{1}{q}\bigg)^q \leq e.\]
Since $Q(\mu, \theta) = 0$ under $f$, and $Q(\mu, \theta) = \frac{1}{n}q\rho^4$ under $g$, it follows from CRI that
\begin{equation*}
R^*(n, \Omega(\beta, \epsilon, b)) \geq c\bigg(\frac{1}{n}q\rho^4\bigg)^2 = cn^{2\epsilon+8b-2}.
\end{equation*}

In fact, when $0<\epsilon\leq\beta<\half$ and $s = \sigma\sqrt{d\log n}$ for some $d>0$, we also have 
\[\inf_{\Qhat}\sup_{\substack{(\mu, \theta): \|\mu\|_0\leq k_n, \|\theta\|_0\leq k_n, \\\|\mu\|_\infty\leq \sigma\sqrt{d\log n}, \|\theta\|_\infty\leq \sigma\sqrt{d\log n}}}E_{(\mu, \theta)}(\Qhat-Q(\mu, \theta))^2 \geq cn^{2\epsilon-2}(\log n)^4.\]
This can be shown by letting $\rho = \sigma\sqrt{\half\min\{d, (1-2\epsilon)\}\log n}$ in \eqref{letrho}.
\end{proof}

\begin{proof}[Proof of Case 3]
The priors used in this case is very different from that considered in the proofs of Case 1 and Case 2.  Let
\[f(x_1, \ldots, x_n, y_1, \ldots, y_n) = \prod_{i=1}^q\psi_s(x_i)\prod_{i=q+1}^n\psi_0(x_i)\prod_{i=1}^q\psi_s(y_i)\prod_{i=q+1}^n\psi_0(y_i),\]
and
\[g(x_1, \ldots, x_n, y_1, \ldots, y_n) = \prod_{i=1}^q\psi_s(x_i)\prod_{i=q+1}^n\psi_0(x_i)\prod_{i=1}^q\psi_{s-\delta}(y_i)\prod_{i=q+1}^n\psi_0(y_i),\]
where $0<\delta<s$.  Note that no mixing is performed in this case.  Instead, we fix the sequence $\mu = (s, \ldots, s, 0, \ldots, 0)$ in both $f$ and $g$, and perturb the nonzero entries of $\theta$ by a small amount $\delta$ in $g$.  This set of priors provides the sharpest rate for the case when signal is strong, i.e., $s=n^b$ is large.  The intuition is that when $s$ is large, estimation of $Q(\mu, \theta)$ is most difficult due to the indistinguishability between $\theta_i = s$ and $\theta_i = s-\delta$, where $\delta\approx 0$.

The chi-square affinity between $f$ and $g$ is given by
\[\int\frac{g^2}{f} = e^{q\delta^2/\sigma^2}.\]
Let $\delta = \sigma/\sqrt{q} = \sigma n^{-\epsilon/2}$.  Then we have 
\[\int\frac{g^2}{f} = e <\infty.\]  
Since $Q(\mu, \theta) = \frac{1}{n}qs^4$ under $f$ and $Q(\mu, \theta) = \frac{1}{n}qs^2(s-\delta)^2$ under $g$, it follows from CRI that 
\begin{align*}
R^*(n, \Omega(\beta, \epsilon, b)) &\geq c\bigg(\frac{1}{n}qs^2\big(s^2-(s-\delta)^2\big)\bigg)^2 \\
&= c\bigg(\frac{1}{n}\sqrt{q}s^3\bigg)^2(1+o(1)) = cn^{\epsilon+6b-2}(1+o(1)).
\end{align*}
\end{proof}

\bibliographystyle{chicago}
\bibliography{Ref.bib}

\newpage
\begin{appendices}
\section{Further Discussions}\label{general}
In the following, we expand on several of the topics discussed in the paper.  In particular, we address the estimation of $Q(\mu, \theta)$ over a more general parameter space than is considered in \eqref{twoSpace}, allowing the signal strengths of $\mu$ and $\theta$ to differ.  We then examine the effect of replacing the $\ell_\infty$-norm constraint on $\mu$ and $\theta$ with the $\ell_2$-norm constraint on the estimation problem.  

\subsection{Estimation of $Q(\mu, \theta)$ with Different Signal Strengths}\label{generalEst}
We consider in Section~\ref{twosample} the estimation of $Q(\mu, \theta) = \frac{1}{n}\sum_{i=1}^n\mu_i^2\theta_i^2$ over the parameter space \eqref{compSpace} where $j_n = k_n = n^\beta$ and $r_n = s_n = n^b$, with $0<\epsilon\leq\beta<\half$ and $b\in\R$.  In this section, we present the estimation result for $Q(\mu, \theta)$ with $j_n = k_n = n^\beta$ but allow $r_n$ and $s_n$ to differ.  Specifically, we consider the parameter space
\begin{align}
\Omega(\beta, \epsilon, a, b) &= \{(\mu, \theta)\in\R^n\times\R^n: \|\mu\|_0\leq k_n, \|\mu\|_\infty\leq r_n, \|\theta\|_0\leq k_n, \|\theta\|_\infty\leq s_n, \nonumber\\
&\qquad\qquad\qquad\qquad\qquad\|\mu\star\theta\|_0\leq q_n\}. \label{paramTwoab}
\end{align}
where $k_n = n^\beta, q_n = n^\epsilon$ with $0<\epsilon\leq\beta<\half$, and $r_n = n^a, s_n = n^b$ with $a, b\in\R$.  

Similar as before, the estimation problem can be divided into three regimes: the sparse regime ($0<\epsilon<\frac{\beta}{2}$), the moderately dense regime ($\frac{\beta}{2}\leq\epsilon\leq\frac{3\beta}{4}$), and the strongly dense regime ($\frac{3\beta}{4}<\epsilon\leq\beta$).  When $\mu$ and $\theta$ have different signal strengths, the minimax rates of convergence for $Q(\mu, \theta)$ exhibit more elaborate phase transitions, though they still bear the familiar form
\[R^*(n, \Omega(\beta, \epsilon, a, b)) := \inf_{\Qhat}\sup_{(\mu, \theta)\in\Omega(\beta, \epsilon, a, b)} E_{(\mu, \theta)}(\Qhat-Q(\mu, \theta))^2 \asymp \gamma_n(\beta, \epsilon, a, b),\]
where $\gamma_n(\beta, \epsilon, a, b)$ is a function of $n$ indexed by $\beta, \epsilon, a$, and $b$. 

For readability, we summarize the corresponding $\gamma_n(\beta, \epsilon, a, b)$ in Table~\ref{sparse} (sparse regime), Table~\ref{dense1} (moderately dense regime), and Table~\ref{dense2} (strongly dense regime), respectively.  In all three tables, we added cases where $r_n = \sigma\sqrt{c\log n}$ (as the transition from $a<0$ to $a>0$), and $s_n = \sigma\sqrt{d\log n}$ (as the transition from $b<0$ to $b>0$).  Moreover, the minimax rates of convergence are attained by the same estimators as before over the respective regimes, as shown by Theorem~\ref{two-inf-sparse-up-gen} and Theorem~\ref{two-inf-dense-up-gen}.  

Although we do not present the result here due to its lengthiness, estimation of $Q(\mu, \theta)$ for the case where no constraint is imposed on either sparsity or signal strength of $\mu$ and $\theta$ can be analyzed analogously provided that the magnitude of the simultaneous sparsity $\epsilon$ is compared to $\alpha$ if $a\geq b$, and to $\beta$ if $b\geq a$, for the characterization of the sparse and dense regimes.  
 
\begin{Theorem}[Sparse Regime]\label{two-inf-sparse-up-gen}
Let $0<\epsilon<\frac{\beta}{2}$ and $0<\beta<\half$.  Then $\Qhat_2$ defined in \eqref{estimator2} with $\tau_n = \log n$ attains the minimax rate of convergence over $\Omega(\beta, \epsilon, a, b)$ for  $(a, b)\in \{(a, b): a\wedge b>0\}$.  On the other hand, $\Qhat_0$ defined in \eqref{estimator0} attains the minimax rate of convergence over $\Omega(\beta, \epsilon, a, b)$ for $(a, b)\in \{(a, b): a\wedge b\leq 0\}$.
\end{Theorem}

\begin{Theorem}[Dense Regime]\label{two-inf-dense-up-gen}
Let $\frac{\beta}{2}\leq\epsilon\leq\beta$ and $0<\beta<\half$.  Then $\Qhat_4$ defined in \eqref{estimator4} with $\tau_n = 4\log n$ attains the minimax rate of convergence over $\Omega(\beta, \epsilon, a, b)$ for $(a, b)\in \{(a, b): a\vee b>0\text{ and } a\wedge b>\frac{\beta-2\epsilon}{4}\}$.  On the other hand, $\Qhat_0$ defined in \eqref{estimator0} attains the minimax rate of convergence over $\Omega(\beta, \epsilon, a, b)$ for $(a, b)\in \{(a, b): a\vee b\leq0\text{ or } a\wedge b\leq\frac{\beta-2\epsilon}{4}\}$.
\end{Theorem}

\begin{Remark}
{\rm
Whenever $r_n = \sigma\sqrt{c\log n}$ and $s_n = \sigma\sqrt{d\log n}$ for some $c, d>0$, the minimax rate of convergence is attained by $\Qhat_0$ and $\Qhat_2$ for $0<\epsilon\leq\beta$, and also by $\Qhat_4$ for $\frac{\beta}{2}<\epsilon\leq\beta$.
} 
\end{Remark}

The shaded regions in the three tables represent the region where $\Qhat_0$ attains the minimax rate of convergence.  Thus, $\{(a, b): a\wedge b\leq 0\}$ is shaded in Table~\ref{sparse}, while $\{(a, b): a\vee b\leq0\text{ or } a\wedge b\leq\frac{\beta-2\epsilon}{4}\}$ is shaded in Table~\ref{dense1} and Table~\ref{dense2}.  Some regions involving $r_n = \sigma\sqrt{c\log n}$ or $s_n = \sigma\sqrt{d\log n}$ are shaded as well, as these represent the boundary of estimation, where we are indifferent in terms of estimating $Q(\mu, \theta)$ by $\Qhat_0$ or $\Qhat_2$ in the sparse regime, and by $\Qhat_0$ or $\Qhat_4$ in the dense regime.

Note that the estimation result for the dense regime turns out to be interesting (and more inspiring) when $r_n$ and $s_n$ can differ.  It seems that estimation is desirable whenever the signal strengths of both sequences barely exceed some small threshold ($a\wedge b>\frac{\beta-2\epsilon}{4}$, but $\beta-2\epsilon\leq 0$ in this case) and at least one sequence has sufficiently strong signal ($a\vee b>0$).  This is in contrast to the sparse regime where estimation is desirable only when the signal strength of both sequences are sufficiently strong ($a\wedge b>0$).  The intuitive explanation is that in the dense regime, knowing that $\mu_i\neq 0$ (because of large $X_i^2$) most often suggests that $\theta_i\neq 0$ too (even if $Y_i^2$ is small), and vice versa, so we cannot afford to estimate $\mu_i^2\theta_i^2$ by 0 with this additional information.  On the contrary, in the sparse regime, knowing that $\mu_i\neq 0$ does not entail much about whether $\theta_i\neq 0$ due to the sparseness of simultaneously nonzero coordinates.  Therefore it is better to estimate $\mu_i^2\theta_i^2$ by 0 unless both $X_i^2$ and $Y_i^2$ are large.

In fact, the minimax rates of convergence for the sparse regime are relatively simple to describe, when $r_n$ is not necessarily equal to $s_n$: 
\begin{equation*}
\gamma_n(\beta, \epsilon, a, b) = \left\{
\begin{array}{ll}
n^{2\epsilon+4a+4b-2} &\text{if }a\wedge b\leq 0, \\
n^{2\epsilon+4a\vee b-2}(\log n)^2 &\text{if }0<a\wedge b\leq\frac{\epsilon}{2}, \\
n^{\epsilon+4a\vee b+2a\wedge b-2} &\text{if }a\wedge b>\frac{\epsilon}{2}.
\end{array}\right.
\end{equation*}
Unfortunately, we do not have such easy representation for the minimax rates of convergence in the dense regime.  Nonetheless, due to the two-dimensional nature of the estimation problem, we find tables useful not only in presenting the minimax rates of convergence but also in illustrating the regions with weak signals (i.e., the shaded regions).

\newgeometry{left=1.8cm,top=1.5cm,right=1.8cm,bottom=1.5cm}
\begin{landscape}
\begin{table}
\begin{centering}
\begin{tabular}{|l| l l l l|}
\hline
& $b\leq 0$ & $s_n=\sigma\sqrt{d\log n}$ & $0<b\leq\frac{\epsilon}{2}$ & $b>\frac{\epsilon}{2}$ \\
\hline
$a\leq 0$ & \cellcolor[gray]{0.85}$n^{2\epsilon+4a+4b-2}$ & \cellcolor[gray]{0.85}$n^{2\epsilon+4a-2}(\log n)^2$ & \cellcolor[gray]{0.85}$n^{2\epsilon+4a+4b-2}$ & \cellcolor[gray]{0.85}$n^{2\epsilon+4a+4b-2}$ \\
$r_n=\sigma\sqrt{c\log n}$ & \cellcolor[gray]{0.85}$n^{2\epsilon+4b-2}(\log n)^2$ & \cellcolor[gray]{0.85}$n^{2\epsilon-2}(\log n)^4$ & \cellcolor[gray]{0.85}$n^{2\epsilon+4b-2}(\log n)^2$ & \cellcolor[gray]{0.85}$n^{2\epsilon+4b-2}(\log n)^2$ \\
$0<a\leq\frac{\epsilon}{2}$ & \cellcolor[gray]{0.85}$n^{2\epsilon+4a+4b-2}$ & \cellcolor[gray]{0.85}$n^{2\epsilon+4a-2}(\log n)^2$ & $n^{2\epsilon+4a\vee b-2}(\log n)^2$ & $n^{2\epsilon+4b-2}(\log n)^2$ \\
$a>\frac{\epsilon}{2}$ & \cellcolor[gray]{0.85}$n^{2\epsilon+4a+4b-2}$ & \cellcolor[gray]{0.85}$n^{2\epsilon+4a-2}(\log n)^2$ & $n^{2\epsilon+4a-2}(\log n)^2$ & $n^{\epsilon+4a\vee b+2a\wedge b-2}$ \\
\hline
\end{tabular}
\vspace{-.75\baselineskip}
\caption{Minimax rates of convergence in the sparse regime: $0 < \epsilon<\frac{\beta}{2}$.}
\label{sparse}
\end{centering}
\end{table}

\begin{table}
\begin{centering}
\begin{tabular}{|l|l l l l l l|}
\hline
& $b\leq\frac{\beta-2\epsilon}{4}$ & $\frac{\beta-2\epsilon}{4}<b\leq 0$ & $s_n=\sigma\sqrt{d\log n}$ & $0<b\leq\frac{2\epsilon-\beta}{4}$ & $\frac{2\epsilon-\beta}{4}<b\leq\frac{\beta-\epsilon}{2}$ & $b>\frac{\beta-\epsilon}{2}$ \\
\hline
$a\leq\frac{\beta-2\epsilon}{4}$ & \cellcolor[gray]{0.85}$n^{2\epsilon+4a+4b-2}$ & \cellcolor[gray]{0.85}$n^{2\epsilon+4a+4b-2}$ & \cellcolor[gray]{0.85}$n^{2\epsilon+4a-2}(\log n)^2$ & \cellcolor[gray]{0.85}$n^{2\epsilon+4a+4b-2}$ & \cellcolor[gray]{0.85}$n^{2\epsilon+4a+4b-2}$ & \cellcolor[gray]{0.85}$n^{2\epsilon+4a+4b-2}$ \\
$\frac{\beta-2\epsilon}{4}<a\leq 0$ & \cellcolor[gray]{0.85}$n^{2\epsilon+4a+4b-2}$ & \cellcolor[gray]{0.85}$n^{2\epsilon+4a+4b-2}$ & \cellcolor[gray]{0.85}$n^{2\epsilon+4a-2}(\log n)^2$ & $\!\begin{aligned}[t]&\max\{n^{\beta+4b-2}, \\&\quad n^{2\epsilon+4a-2}(\log n)^2\}\end{aligned}$ & $n^{\beta+4b-2}$ & $n^{\beta+4b-2}$ \\
$r_n=\sigma\sqrt{c\log n}$ & \cellcolor[gray]{0.85}$n^{2\epsilon+4b-2}(\log n)^2$ & \cellcolor[gray]{0.85}$n^{2\epsilon+4b-2}(\log n)^2$ & $\cellcolor[gray]{0.85}n^{2\epsilon-2}(\log n)^4$ & $n^{2\epsilon-2}(\log n)^4$ & $n^{\beta+4b-2}$ & $n^{\beta+4b-2}$ \\
$0<a\leq\frac{2\epsilon-\beta}{4}$ & \cellcolor[gray]{0.85}$n^{2\epsilon+4a+4b-2}$ & $\!\begin{aligned}[t]&\max\{n^{\beta+4a-2}, \\&\quad n^{2\epsilon+4b-2}(\log n)^2\}\end{aligned}$ & $n^{2\epsilon-2}(\log n)^4$ & $n^{2\epsilon-2}(\log n)^4$ & $n^{\beta+4b-2}$ & $n^{\beta+4b-2}$ \\
$\frac{2\epsilon-\beta}{4}<a\leq\frac{\beta-\epsilon}{2}$ & \cellcolor[gray]{0.85}$n^{2\epsilon+4a+4b-2}$ & $n^{\beta+4a-2}$ & $n^{\beta+4a-2}$ & $n^{\beta+4a-2}$ & $n^{\beta+4a\vee b-2}$ & $n^{\beta+4b-2}$ \\
$a>\frac{\beta-\epsilon}{2}$ & \cellcolor[gray]{0.85}$n^{2\epsilon+4a+4b-2}$ & $n^{\beta+4a-2}$ & $n^{\beta+4a-2}$ & $n^{\beta+4a-2}$ & $n^{\beta+4a-2}$ & $n^{\epsilon+4a\vee b+2a\wedge b-2}$\\
\hline
\end{tabular}
\vspace{-.75\baselineskip}
\caption{Minimax rates of convergence in the moderately dense regime: $\frac{\beta}{2} \leq \epsilon \leq \frac{3\beta}{4}$.  In this case, we have $\frac{2\epsilon-\beta}{4}\leq\frac{\beta-\epsilon}{2}$.}
\label{dense1}
\end{centering}
\end{table}

\begin{table}
\begin{centering}
\begin{tabular}{|l|l l l l l l|}
\hline
& $b\leq\frac{\beta-2\epsilon}{4}$ & $\frac{\beta-2\epsilon}{4}<b\leq 0$ & $s_n=\sigma\sqrt{d\log n}$ & $0<b\leq\frac{\beta-\epsilon}{2}$ & $\frac{\beta-\epsilon}{2}<b\leq\frac{2\epsilon-\beta}{4}$ & $b>\frac{2\epsilon-\beta}{4}$ \\
\hline
$a\leq\frac{\beta-2\epsilon}{4}$ & \cellcolor[gray]{0.85}$n^{2\epsilon+4a+4b-2}$ & \cellcolor[gray]{0.85}$n^{2\epsilon+4a+4b-2}$ & \cellcolor[gray]{0.85}$n^{2\epsilon+4a-2}(\log n)^2$ & \cellcolor[gray]{0.85}$n^{2\epsilon+4a+4b-2}$ & \cellcolor[gray]{0.85}$n^{2\epsilon+4a+4b-2}$ & \cellcolor[gray]{0.85}$n^{2\epsilon+4a+4b-2}$ \\
$\frac{\beta-2\epsilon}{4}<a\leq 0$ & \cellcolor[gray]{0.85}$n^{2\epsilon+4a+4b-2}$ & \cellcolor[gray]{0.85}$n^{2\epsilon+4a+4b-2}$ & \cellcolor[gray]{0.85}$n^{2\epsilon+4a-2}(\log n)^2$ & $\!\begin{aligned}[t]&\max\{n^{\beta+4b-2}, \\&\quad n^{2\epsilon+4a-2}(\log n)^2\}\end{aligned}$ & $\!\begin{aligned}[t]&\max\{n^{\beta+4b-2}, \\&\quad n^{2\epsilon+4a-2}(\log n)^2\}\end{aligned}$ & $n^{\beta+4b-2}$ \\
$r_n=\sigma\sqrt{c\log n}$ & \cellcolor[gray]{0.85}$n^{2\epsilon+4b-2}(\log n)^2$ & \cellcolor[gray]{0.85}$n^{2\epsilon+4b-2}(\log n)^2$ & \cellcolor[gray]{0.85}$n^{2\epsilon-2}(\log n)^4$ & $n^{2\epsilon-2}(\log n)^4$ & $n^{2\epsilon-2}(\log n)^4$ & $n^{\beta+4b-2}$ \\
$0<a\leq\frac{\beta-\epsilon}{2}$ & \cellcolor[gray]{0.85}$n^{2\epsilon+4a+4b-2}$ & $\!\begin{aligned}[t]&\max\{n^{\beta+4a-2}, \\&\quad n^{2\epsilon+4b-2}(\log n)^2\}\end{aligned}$ & $n^{2\epsilon-2}(\log n)^4$ & $n^{2\epsilon-2}(\log n)^4$ & $n^{2\epsilon-2}(\log n)^4$ & $n^{\beta+4b-2}$ \\
$\frac{\beta-\epsilon}{2}<a\leq\frac{2\epsilon-\beta}{4}$ & \cellcolor[gray]{0.85}$n^{2\epsilon+4a+4b-2}$ & $\!\begin{aligned}[t]&\max\{n^{\beta+4a-2}, \\&\quad n^{2\epsilon+4b-2}(\log n)^2\}\end{aligned}$ & $n^{2\epsilon-2}(\log n)^4$ & $n^{2\epsilon-2}(\log n)^4$ & $\!\begin{aligned}[t]&\max\{n^{2\epsilon-2}(\log n)^4, \\
&\quad n^{\epsilon+4a\vee b+2a\wedge b-2}\}\end{aligned}$ & $n^{\epsilon+2a+4b-2}$ \\
$a>\frac{2\epsilon-\beta}{4}$ & \cellcolor[gray]{0.85}$n^{2\epsilon+4a+4b-2}$ & $n^{\beta+4a-2}$ & $n^{\beta+4a-2}$ & $n^{\beta+4a-2}$ & $n^{\epsilon+4a+2b-2}$ & $n^{\epsilon+4a\vee b+2a\wedge b-2}$\\
\hline
\end{tabular}
\vspace{-.75\baselineskip}
\caption{Minimax rates of convergence in the strongly dense regime: $\frac{3\beta}{4} < \epsilon \leq \beta$.  In this case, we have $\frac{\beta-\epsilon}{2}<\frac{2\epsilon-\beta}{4}$.}
\label{dense2}
\end{centering}
\end{table}
\end{landscape}
\restoregeometry

\subsection{Estimation of $Q(\mu, \theta)$ with $\ell_2$-norm Constraint on $\mu$ and $\theta$}
In Section~\ref{twosample}, we consider the estimation of $Q(\mu, \theta)$ over the the parameter space defined in \eqref{twoSpace}.  Despite our explicit interest in the analysis of parameter space with rare signals, our choice of the $\ell_\infty$-norm in encoding signal strength seems somewhat arbitrary.  One might ask if the estimation problem exhibits different behavior had we chosen a different norm.  We have a partial answer for this.  Consider the following family of parameter spaces where signal strength is expressed in terms of $\ell_2$-norm rather than $\ell_\infty$-norm:
\begin{align*}
\tilde{\Omega}(\alpha, \beta, \epsilon, \tilde{a}, \tilde{b})  &= \{(\mu, \theta)\in\R^n\times\R^n: \|\mu\|_0\leq j_n, \|\mu\|_2\leq \tilde{r}_n, \|\theta\|_0\leq k_n, \|\theta\|_2\leq \tilde{s}_n, \\
&\qquad\qquad\qquad\qquad\qquad\|\mu\star\theta\|_0\leq q_n\},
\end{align*}
where $j_n = n^\alpha, k_n = n^\beta, q_n = n^\epsilon, 0<\epsilon\leq\alpha\wedge\beta<\half$ and $\tilde{r}_n = n^{\tilde{a}}, \tilde{s}_n = n^{\tilde{b}}, \tilde{a}, \tilde{b}\in\R$.  Due to the ease of presentation of estimation result in this case, we allow different levels of both sparsity and signal strength for $\mu$ and $\theta$ for greater generality.  The reader is free to compare the estimation results in this case with that presented in Section~\ref{generalEst}, where we also allow $r_n$ and $s_n$ to differ but for the $\ell_\infty$-norm constraint.  It turns out that the estimation problem in the case of $\ell_2$-norm constraint is in some sense degenerate and not as meaningful when compared to that with $\ell_\infty$-norm constraint.  We have
\[\inf_{\Qhat}\sup_{(\mu, \theta)\in \tilde{\Omega}(\alpha, \beta, \epsilon, \tilde{a}, \tilde{b})}E_{(\mu, \theta)}(\Qhat-Q(\mu, \theta))^2 \asymp \gamma_n(\alpha, \beta, \epsilon, \tilde{a}, \tilde{b}),\]
where
\begin{equation*}
\gamma_n(\alpha, \beta, \epsilon, \tilde{a}, \tilde{b}) = \left\{
\begin{array}{ll}
n^{4\tilde{a}+4\tilde{b}-2} &\text{if }\tilde{a}\wedge \tilde{b}\leq 0, \\
n^{4a\vee b+2a\wedge b-2} &\text{if }\tilde{a}\wedge \tilde{b}>0.
\end{array}\right.
\end{equation*}

Note that the minimax rates of convergence here are independent of $\alpha, \beta$, and $\epsilon$.  A lower bound analysis reveals that the worst case for the estimation problem happens when the signal is concentrated at one coordinate, i.e., $\mu_i = \tilde{r}_n, \theta_i = \tilde{s}_n$ for some $i$ and $\mu_j = \theta_j = 0$ for all $j\neq i$.  The estimator $\Qhat_0 = 0$ is optimal when $a\wedge b\leq 0$, while the estimator $\Qhat_2$ defined in \eqref{estimator2} is optimal when $a\wedge b>0$.  

On the contrary, for the estimation of $Q(\theta)$ in the one-sequence case, if we consider the following family of parameter spaces
\[\tilde{\Theta}(\beta, \tilde{b})  = \{\theta\in\R^n: \|\theta\|_0\leq k_n, \|\theta\|_2\leq \tilde{s}_n\},\]
where $k_n = n^\beta, 0<\beta<1$ and $\tilde{s}_n = n^{\tilde{b}}, \tilde{b}\in\R$, then the minimax rates of estimation exhibit similar behavior as in the case with $\ell_\infty$-norm constraint.  As before, we have
\[\inf_{\Qhat}\sup_{\theta\in \tilde{\Theta}(\beta, b)}E_\theta(\Qhat-Q(\theta))^2 \asymp \gamma_n(\beta, \tilde{b}),\]
 where
\begin{equation}
\gamma_n(\beta, \tilde{b}) = \left\{
\begin{array}{ll}
n^{4\tilde{b}-2} &\text{if }\tilde{b}\leq \frac{\beta}{2}, \\
n^{2\beta-2}(\log n)^2 &\text{if }\frac{\beta}{2} < \tilde{b} \leq \beta, \\
n^{2\tilde{b}-2} &\text{if }\tilde{b}>\beta, \\
\end{array}\right. \label{2onesamplerateS}
\end{equation}
when $0<\beta<\half$ and

\begin{equation}
\gamma_n(\beta, \tilde{b}) = \left\{
\begin{array}{ll}
n^{4\tilde{b}-2} &\text{if }\tilde{b}\leq \frac{1}{4}, \\
n^{-1} &\text{if }\frac{1}{4} < \tilde{b} \leq \half, \\
n^{2\tilde{b}-2} &\text{if }\tilde{b}>\half, \label{2onesamplerateD}
\end{array}\right.
\end{equation}
when $\half\leq\beta\leq 1$.  Letting $\tilde{b} = \frac{\beta}{2}+b$ in \eqref{2onesamplerateS} and \eqref{2onesamplerateD} yields the rate given in \eqref{onesamplerateS} and \eqref{onesamplerateD}, respectively.  The idea here is to link the $\ell_2$-norm constraint to the $\ell_\infty$-norm constraint via $\tilde{s}_n^2 = k_ns_n^2$, since then we essentially impose the same amount of constraint on the maximal value of $Q(\theta) = \frac{1}{n}\sum\theta_i^2$. 

Nevertheless, in the two-sequence case, the $\ell_2$-norm constraint translates to $Q(\mu, \theta) = \frac{1}{n}\sum\mu_i^2\theta_i^2 \leq \frac{1}{n}\tilde{r}_n^2\tilde{s}_n^2$ (equality holds if $\mu_i = \tilde{r}_n, \theta_i = \tilde{s}_n$ for some $i$ and $\mu_j = \theta_j = 0$ for all $j\neq i$), which is quite different from $Q(\mu, \theta) = \frac{1}{n}\sum\mu_i^2\theta_i^2 \leq \frac{1}{n}q_nr_n^2s_n^2$ as translated by the $\ell_\infty$-norm constraint.  It does not seem sensible to link the two constrains via the relationship $\tilde{r}_n^2 = j_nr_n^2$ and $\tilde{s}_n^2 = k_ns_n^2$, since these are specific to the value of $Q(\mu) = \frac{1}{n}\sum\mu_i^2$ and $Q(\theta) = \frac{1}{n}\sum\theta_i^2$, yet it is not immediately clear how the two constraints on the value of $Q(\mu, \theta)$ should be related.   

The disparity in the behavior of minimax rates of convergence between the one-sequence and two-sequence case under the $\ell_2$-norm constraint seems to suggest that the intrinsic nature of the problem of estimating quadratic functional for the two cases are somewhat different, at least for the specific family of parameter spaces that we consider.  To better spell out the distinctions and similarities between the two cases, a more refined analysis of the estimation problem under $\ell_p$-norm constraint on the signal strength for $p\in(0, \infty]$ seems necessary.  We suspect a smooth and monotone relationship between $p$ and the degeneracy of the estimation problem in both the one-sequence and the two-sequence cases, though the corresponding range of $p$ where transition occurs might be different.  In fact, it is easy to check that for the one-sequence case, estimation of quadratic functional is degenerate when $0<p\leq 1$, and is meaningful when $p\geq 2$.  On the contrary, for the two-sequence case, estimation of quadratic functional is degenerate when $0<p\leq 2$.  An enumeration of results for all $p\in(0, \infty]$ is beyond the scope of the paper.

\section{Additional Proofs}\label{prooftwosample} 
In this section, we present the proofs of Theorem~\ref{two-inf-dense-up}, Theorem~\ref{two-inf-dense}, and Theorem~\ref{two-det-Q}.  Proofs of lower bounds are given in Section~\ref{two-lower}, followed by proofs of upper bounds in Section~\ref{two-upper}.

\subsection{Proofs of Lower Bounds}\label{two-lower}
We first prove Theorem~\ref{two-inf-dense}, which constitute the lower bound for the estimation rate of $Q(\mu, \theta)$ in the dense regime.  We then prove the minimax lower bound for the hypothesis testing problem considered in Theorem~\ref{two-det-Q}, to characterize the undetectable region.  We begin with some technical tools for establishing lower bounds.

\subsubsection{General Tools}\label{tools}
Let $\mP$ be a set of probability measures on a measurable space $(\mX, \mA)$, and let $\theta: \mP\longrightarrow\R$.  For $P_f, P_g\in\mP$, let $\theta_f = \theta(P_f), \theta_g = \theta(P_g)$, and let $f, g$ denote the density of $P_f, P_g$ with respect to some dominating measure $u$.  The chi-square affinity between $P_f$ and $P_g$ is defined as
\[\xi = \xi(P_f, P_g) = \int\frac{g^2}{f}\ du.\]
In particular, for Gaussian distributions, we have
\[\xi(N(\theta_0, \sigma^2), N(\theta_1, \sigma^2)) = e^{(\theta_1-\theta_0)^2/\sigma^2}.\]

Throughout, the proof of lower bounds is established by constructing two priors which have small chi-square distance but a large difference in the expected values of the resulting quadratic functionals, and then applying the Constrained Risk Inequality (CRI) in \cite{BrownLow1996}.  Essentially, CRI says that if $P_f$ and $P_g$ are such that $\theta_f, \theta_g\in\Theta$, the parameter space of estimation, with $\xi = \xi(P_f, P_g)<\infty$, then for any estimator $\delta$ of $\theta = \theta(P)\in\Theta$ based on the random variable $X$ with distribution $P$, we have
\[\sup_{\theta\in\Theta} E_\theta(\delta(X)-\theta)^2 \geq \frac{(\theta_g-\theta_f)^2}{(1+\xi^{1/2})^2}.\]
It follows that to establish lower bound for estimation rate, it suffices to find $P_f$ and $P_g$ such that $(\theta_g-\theta_f)^2$ is as large as possible subject to $\xi(P_f, P_g)<\infty$.

Hereafter, we omit the subscripts $n$ in $k_n, q_n, s_n$ and $\tau_n$ that signifies their dependence on the sample size.  We denote by $\psi_\mu$ the density of a Gaussian distribution with mean $\mu$ and variance $\sigma^2$, and we denote by $\ell(n, k)$ the class of all subsets of $\{1, \ldots, n\}$ of $k$ distinct elements.  Also, we remind the readers that
\[R^*(n, \Omega(\beta, \epsilon, b)) = \inf_{\Qhat}\sup_{(\mu, \theta)\in\Omega(\beta, \epsilon, b)}E_{(\mu, \theta)}(\Qhat-Q(\mu, \theta))^2.\]
Finally, $c$ and $C$ denote constants that may vary for each occurrence.

\subsubsection{Proof of Theorem~\ref{two-inf-dense}}\label{estproof}
To prove Theorem~\ref{two-inf-dense}, it is sufficient to show that for $0<\beta<\half$,
\begin{equation*}
\gamma_n(\beta, \epsilon, b) \geq \left\{
\begin{array}{llll}
n^{2\epsilon+8b-2} &\text{if }b\leq 0, &\text{for }0<\epsilon\leq\beta, & \qquad\qquad(\text{Case } 2) \\
n^{\epsilon+6b-2} &\text{if }b>0, &\text{for } 0<\epsilon\leq\beta, & \qquad\qquad(\text{Case } 3) \\
n^{\beta+4b-2} &\text{if }b>0, &\text{for }\frac{\beta}{2}\leq\epsilon\leq\beta, & \qquad\qquad(\text{Case } 4) \\
n^{2\epsilon-2}(\log n)^4 &\text{if }b>0, &\text{for }0<\epsilon\leq\beta. & \qquad\qquad(\text{Case } 5) \\
\end{array}\right.
\end{equation*}
The proof of Case 2 and Case 3 can be found in Section~\ref{sparse-lower-proof}, hence we will only provide proofs of Case 4 and Case 5 below.  For individual regions in $\{(\beta, \epsilon, b): \frac{\beta}{2}\leq\epsilon\leq\beta<\half, b\in\R\}$, the minimax rate of convergence is obtained as the sharpest rate among all cases in which the region belongs to.  For instance, the region $\{(\beta, \epsilon, b): \frac{3\beta}{4}<\epsilon\leq\beta<\half, b>\frac{\epsilon}{6}\}$ is included in Case 3, Case 4 and Case 5, hence $\gamma_n(\beta, \epsilon, b) \geq \max\{n^{\epsilon+6b-2}, n^{\beta+4b-2}, n^{2\epsilon-2}(\log n)^4\} = n^{\epsilon+6b-2}$.

\begin{proof}[Proof of Case 4]
The proof of Case 4 is very similar to the proof of Case 1, besides that a slightly different mixture prior $g$ is employed.  Let
\[f(x_1, \ldots, x_n, y_1, \ldots, y_n) = \prod_{i=1}^k \psi_s(x_i)\prod_{i=k+1}^n\psi_0(x_i)\prod_{i=1}^n \psi_0(y_i).\]
For $I\in\ell(k, q)$, let
\begin{align*}
&g_I(x_1, \ldots, x_n, y_1, \ldots, y_n) \\
&= \prod_{i=1}^k\psi_s(x_i)\prod_{i=k+1}^n\psi_0(x_i)\prod_{i=1}^k\bigg[\half\psi_{\theta_i}(y_i)+\half\psi_{-\theta_i}(y_i)\bigg]\prod_{i=k+1}^n \psi_0(y_i),
\end{align*}
where $\theta_i = \rho\1(i\in I)$ with $\rho>0$, and let
\[g = \frac{1}{\binom{k}{q}}\sum_{I\in\ell(k, q)} g_I.\]
Note that in constructing $g$, mixing is done not only over all possible subsets $\ell(k, q)$ but also over the signs of $\theta_i$'s.  This has largely to do with the intuition that when signal is abundant, uncertainty about the signs of $\theta_i$'s further increase the difficulty of the estimation problem.  That being said, mixing without sign flips (i.e., simply use the priors $f$ and $g$ as given in the proof of Case 1) does not give us the tightest lower bound.  Similar to Case 1, keeping $\mu = (s, \ldots, s, 0, \ldots, 0)$ the same in both $f$ and $g$ essentially reduces the two-sequence problem to a one-sequence problem.  Our choice of priors is equivalent to having only one Gaussian mean sequence of length $k$ with $q$ nonzero entries --- thus the correspondence between the dense regime in the two-sequence case ($q\geq\sqrt{k}$) and the dense regime in the one-sequence case ($k\geq\sqrt{n}$). 

Again, the chi-square affinity between $f$ and $g$ has the form \eqref{mixprior}, where for $I, J\in\ell(k, q)$ with $m = \text{Card}(I\cap J)$, 
\begin{align*}
\int\frac{g_Ig_J}{f} &= \prod_{i=1}^k\int \frac{[\half\psi_{\rho\1(i\in I)}(y_i) + \half\psi_{-\rho\1(i\in I)}(y_i)][\half\psi_{\rho\1(i\in J)}(y_i) + \half\psi_{-\rho\1(i\in J)}(y_i)]}{\psi_0(y_i)}\ dy_i \\
&= \prod_{i=1}^k \int \frac{1}{4}\bigg\{\frac{\psi_{\rho\1(i\in I)}(y_i)\psi_{\rho\1(i\in J)}(y_i)}{\psi_0(y_i)} + \frac{\psi_{-\rho\1(i\in I)}(y_i)\psi_{-\rho\1(i\in J)}(y_i)}{\psi_0(y_i)} \\
&\qquad\qquad\qquad + \frac{\psi_{\rho\1(i\in I)}(y_i)\psi_{-\rho\1(i\in J)}(y_i)}{\psi_0(y_i)} + \frac{\psi_{-\rho\1(i\in I)}(y_i)\psi_{\rho\1(i\in J)}(y_i)}{\psi_0(y_i)}\bigg\}\ dy_i \\
&= \prod_{i\in I\cap J} \frac{1}{4}\bigg[\int\frac{\psi_\rho^2(y_i)}{\psi_0(y_i)} + \int\frac{\psi_{-\rho}^2(y_i)}{\psi_0(y_i)} + 2\int\frac{\psi_\rho(y_i)\psi_{-\rho}(y_i)}{\psi_0(y_i)}\bigg]\prod_{i\in I^c\cup J^c}1 \\
&= \prod_{i\in I\cap J}\frac{1}{2}\big[\exp(\rho^2/\sigma^2) + \exp(-\rho^2/\sigma^2)\big] \\
&= \cosh(\rho^2/\sigma^2)^m.
\end{align*}
It follows that
\[\int\frac{g^2}{f} = E[\cosh(\rho^2/\sigma^2)^M],\]
where $M$ follows hypergeometric distribution as in \eqref{hyperprob}.  Since $M$ coincides in distribution with the conditional expectation $E(\tilde{M}|\mathcal{B})$ where $\tilde{M}$ is a Binomial($q$, $\frac{q}{k}$) random variable and $\mathcal{B}$ is a suitable $\sigma$-algebra \citep{Aldous1983}, with Jensen's inequality, we get
\[\int\frac{g^2}{f} \leq E[\cosh(\rho^2/\sigma^2)^{\tilde{M}}] = \bigg(1+\frac{q}{k}[\cosh(\rho^2/\sigma^2)-1]\bigg)^q.\]
Since $\cosh(x) = \frac{1}{2}(e^x+e^{-x}) = 1+\frac{x^2}{2} + o(x^2)$ when $x\approx 0$, taking $x = \rho^2/\sigma^2$ with $\rho = (\frac{k}{q^2})^{1/4}$ yields
\[\int\frac{g^2}{f} \leq \bigg(1+\frac{1}{2\sigma^4q}\bigg)^q < \infty.\]
Since $Q(\mu, \theta) = 0$ under $f$ and $Q(\mu, \theta) = \frac{1}{n}qs^2\rho^2$ under $g$, it follows from CRI that 
\[R^*(n, \Omega(\beta, \epsilon, b)) \geq c\bigg(\frac{1}{n}qs^2\rho^2\bigg)^2 = cn^{\beta+4b-2}.\]
\end{proof}

\begin{proof}[Proof of Case 5]
Let $f$ and $g$ be as given in the proof of Case 2 in Section~\ref{sparse-lower-proof}, and take $\rho = \sigma\sqrt{\half(1-2\epsilon)\log n}$ in \eqref{letrho}. 
It follows that when $n$ is sufficiently large, 
\[e^{2\rho^2/\sigma^2} = n^{1-2\epsilon} = \frac{n}{q^2},\]
hence
\[\int\frac{g^2}{f} \leq \bigg(1+\frac{1}{q}\bigg)^q \leq e.\]
Since $Q(\mu, \theta) = 0$ under $f$, and $Q(\mu, \theta) = \frac{1}{n}q\rho^4$ under $g$, it follows from CRI that
\begin{equation*}
R^*(n, \Omega(\beta, \epsilon, b)) \geq c\bigg(\frac{1}{n}q\rho^4\bigg)^2 = cn^{2\epsilon-2}(\log n)^4.
\end{equation*}
\end{proof}

\subsubsection{Proof of Lower Bound in Theorem~\ref{two-det-Q}}\label{det-low-proof}
Consider testing between
\[H_0: \theta\in\Theta_0(n), \qquad H_1: \theta\in\Theta_1(n).\]
Let $f$ be the density of a prior supported on $\Theta_0(n)$, and let $g$ be the density of a prior supported on $\Theta_1(n)$.  Using the relationship
\[\inf_\psi\Big[\sup_{\theta\in\Theta_0(n)}E_\theta(\psi) +\sup_{\theta\in\Theta_1(n)}E_\theta(1-\psi)\Big] \geq \inf_\psi[E_f(\psi)+E_g(1-\psi)] = 1 - \half\int |g-f|, \]
and
\[\int |g-f| = \int\frac{|g-f|}{f}f \leq \bigg(\int\frac{(g-f)^2}{f^2}f\bigg)^{1/2} \leq\bigg(\int\frac{g^2}{f}-1\bigg)^{1/2},\]
it follows that to show
\[\inf_\psi\Big[\sup_{\theta\in\Theta_0(n)}E_\theta(\psi) +\sup_{\theta\in\Theta_1(n)}E_\theta(1-\psi)\Big] \longrightarrow 1,\]
it suffices to establish that $\int\frac{g^2}{f}\longrightarrow 1$.  

Below, we will use this idea to establish minimax lower bound for characterizing the undetectable region in the hypothesis testing problem \eqref{twoTestingA}.  We divide the proof of lower bound in Theorem~\ref{two-det-Q} into two cases: the sparse regime where $0<\epsilon<\frac{\beta}{2}$, and the dense regime where $\frac{\beta}{2}\leq\epsilon\leq\beta$.

\begin{proof}[Proof for sparse regime]
Suppose that $0<\epsilon<\frac{\beta}{2}$ and $a\wedge b\leq 0$.  We will show that there exists $f$ and $g$ supported on $\Omega_0(\beta, a, b)$ and $\Omega_1(\beta, \epsilon, a, b)$, respectively, such that $\int\frac{g^2}{f}\longrightarrow 1$.  

Without loss of generality, assume $a\geq b$ with $b\leq 0$.  Let 
\[f(x_1, \ldots, x_n, y_1, \ldots, y_n) = \prod_{i=1}^k \psi_r(x_i)\prod_{i=k+1}^n\psi_0(x_i)\prod_{i=1}^n \psi_0(y_i).\]
For $I\in\ell(k, q)$, let
\[g_I(x_1, \ldots, x_n, y_1, \ldots, y_n) = \prod_{i=1}^k \psi_r(x_i)\prod_{i=k+1}^n\psi_0(x_i)\prod_{i=1}^k\psi_{\theta_i}(y_i)\prod_{i=k+1}^n \psi_0(y_i),\]
where $\theta_i = s\1(i\in I)$, and let
\[g = \frac{1}{\binom{k}{q}}\sum_{I\in\ell(k, q)} g_I.\]
The calculation from the proof of Case 1 in Section~\ref{sparse-lower-proof} shows that
\[\int\frac{g^2}{f} \leq \bigg(1 - \frac{q}{k} + \frac{q}{k}e^{s^2/\sigma^2}\bigg)^q. \]
Since $s=n^b$ with $b\leq 0$, it follows that for $n$ sufficiently large, 
\[\int\frac{g^2}{f} \leq \bigg(1 + \frac{q}{k}e^{n^{2b}/\sigma^2}\bigg)^q \longrightarrow 1.\]
\end{proof}

\begin{proof}[Proof for dense regime]
Suppose that $\frac{\beta}{2}\leq\epsilon\leq\beta$, and either (1) $a\wedge b<\frac{\beta-2\epsilon}{4}$ or (2) $a\vee b\leq 0$.  We will show that in either scenario, there exists $f$ and $g$ supported on $\Omega_0(\beta, a, b)$ and $\Omega_1(\beta, \epsilon, a, b)$, respectively, such that $\int\frac{g^2}{f}\longrightarrow 1$. 

\begin{enumerate}[(1)]
\item Suppose that $a\wedge b<\frac{\beta-2\epsilon}{4}$.  Without loss of generality, assume $a\geq b$ with $b<\frac{\beta-2\epsilon}{4}$.  Let
\[f(x_1, \ldots, x_n, y_1, \ldots, y_n) = \prod_{i=1}^k \psi_r(x_i)\prod_{i=k+1}^n\psi_0(x_i)\prod_{i=1}^n \psi_0(y_i).\]
For $I\in\ell(k, q)$, let
\begin{align*}
&g_I(x_1, \ldots, x_n, y_1, \ldots, y_n) \\
&= \prod_{i=1}^k\psi_r(x_i)\prod_{i=k+1}^n\psi_0(x_i)\prod_{i=1}^k\bigg[\half\psi_{\theta_i}(y_i)+\half\psi_{-\theta_i}(y_i)\bigg]\prod_{i=k+1}^n \psi_0(y_i),
\end{align*}
where $\theta_i = s\1(i\in I)$, and let
\[g = \frac{1}{\binom{k}{q}}\sum_{I\in\ell(k, q)} g_I.\]
The calculation from the proof of Case 4 in Section~\ref{estproof} shows that
\[\int\frac{g^2}{f} \leq \bigg(1+\frac{q}{k}[\cosh(s^2/\sigma^2)-1]\bigg)^q.\]
Since $\cosh(x) = \frac{1}{2}(e^x+e^{-x}) = 1+\frac{x^2}{2}+o(x^2)$ when $x\approx 0$, plug in $x = s^2/\sigma^2$, it follows that for $s=n^b$ with $b<\frac{\beta-2\epsilon}{4}$, 
\[\int\frac{g^2}{f} \leq \bigg(1+\frac{q}{k}\frac{s^4}{2\sigma^4}\bigg)^q = \bigg(1+\frac{1}{2\sigma^4}n^{4b+\epsilon-\beta}\bigg)^{n^\epsilon} \longrightarrow 1.\]

\item  Suppose that $a\vee b\leq 0$.  Let 
\[f(x_1, \ldots, x_n, y_1, \ldots, y_n) = \prod_{i=1}^n \psi_0(x_i)\prod_{i=1}^n \psi_0(y_i)\]
For $I\in\ell(n, q)$, let
\[g_I(x_1, \ldots, x_n, y_1, \ldots, y_n) = \prod_{i=1}^n\psi_{\mu_i}(x_i)\prod_{i=1}^n\psi_{\theta_i}(y_i),\]
where $\mu_i = r\1(i\in I)$, $\theta_i = s\1(i\in I)$, and let
\[g = \frac{1}{\binom{n}{q}}\sum_{I\in\ell(n, q)} g_I.\]
Similar calculation as in the proof of Case 2 in Section~\ref{sparse-lower-proof} shows that
\[\int\frac{g^2}{f} \leq \bigg(1 - \frac{q}{n} + \frac{q}{n}e^{(r^2+s^2)/\sigma^2}\bigg)^q.\] 
Since $a\vee b\leq 0$, for $r = n^a$ and $s = n^b$, we have
\[\int\frac{g^2}{f} \leq \bigg(1 - \frac{q}{n} + \frac{q}{n}e^{(n^{2a}+n^{2b})/\sigma^2}\bigg)^q\longrightarrow 1.\] 
\end{enumerate}
\end{proof}

\subsection{Proofs of Upper Bounds} \label{two-upper}
In this section, we provide proof of Theorem~\ref{two-inf-dense-up} and proof of upper bound in Theorem~\ref{two-det-Q}.  To prove Thorem~\ref{two-inf-dense-up}, we compute upper bound for the the supremum risk of the estimator $\Qhat_4$ (defined in \eqref{estimator4}) over the the parameter space $\Omega(\beta, \epsilon, b)$ introduced in \eqref{twoSpace}.  We will see that it matches the lower bound derived in Section~\ref{two-lower} when $b>0$.  We then prove the upper bound in Theorem~\ref{two-det-Q}, to show that the proposed tests asymptotically separate the null hypothesis from the alternative hypothesis over the detectable regions.  In the rest, we denote by $\phi(z), \Phi(z) = P(Z\leq z)$, and $\tPhi(z) = 1-\Phi(z)$ the density, the cumulative distribution function, and the survival function of a standard normal random variable $Z$, respectively. 

It is trivial to see that for $\Qhat_0 = 0$,
\begin{align}
\sup_{(\mu, \theta)\in \Omega(\beta, \epsilon, b)} E_{(\mu, \theta)}(\Qhat-Q(\mu, \theta))^2 &= \sup_{(\mu, \theta)\in \Omega(\beta, \epsilon, b)}\bigg(\frac{1}{n}\sum_{i=1}^n\mu_i^2\theta_i^2\bigg)^2 \nonumber\\
&= q^2s^8n^{-2} = n^{2\epsilon+8b-2}. \label{Q0rateTwo}
\end{align}

\subsubsection{Proof of Theorem~\ref{two-inf-dense-up}}\label{dense-upper-proof}
The proof of Theorem~\ref{two-inf-dense-up} involves a careful analysis of the bias and variance of the proposed estimator $\Qhat_4$.  We will use the following two lemmas.
\begin{Lemma}\label{biasestimator4}
Let $X\sim N(\mu, \sigma^2), Y\sim N(\theta, \sigma^2)$.  Set $\eta = E_{(0, 0)}[(X^2-\sigma^2)(Y^2-\sigma^2)\1(X^2\vee Y^2>\sigma^2\tau)]$, where the expectation is taken under $\mu=\theta=0$.  Then
\[\eta = -4\sigma^4\tau\phi^2(\tau^{1/2}),\]
and for $\tau\geq 1$,
\begin{align*}
&\big|E[(X^2-\sigma^2)(Y^2-\sigma^2)\1(X^2\vee Y^2>\sigma^2\tau)]-\eta-\mu^2\theta^2\big| \\
&\leq \min\{\mu^2, 3\sigma^2\tau\}\min\{\theta^2, 3\sigma^2\tau\} + 2\sigma^2\tau^{1/2}\phi(\tau^{1/2})\min\{\mu^2, 3\sigma^2\tau\} \\
&\qquad+ 2\sigma^2\tau^{1/2}\phi(\tau^{1/2})\min\{\theta^2, 3\sigma^2\tau\}.
\end{align*}
\end{Lemma}

\begin{Lemma}\label{varestimator4}
Let $X\sim N(\mu, \sigma^2), Y\sim N(\theta, \sigma^2)$.  Then for $\tau\geq 1$,
\begin{align*}
&\Var[(X^2-\sigma^2)(Y^2-\sigma^2)\1(X^2\vee Y^2>\sigma^2\tau)] \\
&\leq \left\{
\begin{array}{ll}
2d^{1/2}\tPhi(\tau^{1/2})^{1/2} &\text{if }\mu = \theta = 0, \\
4\sigma^2\mu^4\theta^2 + 4\sigma^2\mu^2\theta^4 + 16\sigma^4\mu^2\theta^2 + 2\sigma^4\mu^4 + 2\sigma^4\theta^4 \\
\qquad + 8\sigma^6\mu^2 + 8\sigma^6\theta^2 + 4\sigma^8 + 8\sigma^4\mu^2\theta^2\tau^2 &\text{otherwise}, 
\end{array}\right.
\end{align*}
where $d = E_{(0, 0)}[(X^2-\sigma^2)^4(Y^2-\sigma^2)^4]$.
\end{Lemma}
For clarity, we relegate the proofs of Lemma~\ref{biasestimator4} and Lemma~\ref{varestimator4} to the end of Section~\ref{two-upper}.

\begin{proof}[Proof of Theorem~\ref{two-inf-dense-up}]
We first compute the bias of the estimator $\Qhat_4$ defined in \eqref{estimator4}.  It follows from Lemma~\ref{biasestimator4} that for all $(\mu, \theta)\in \Omega(\beta, \epsilon, b)$ and $\tau\geq 1$, we have
\begin{align*}
&\big|E_{(\mu, \theta)}(\Qhat_4) - Q(\mu, \theta)\big| \\
&\leq \frac{1}{n}\sum_{i=1}^n\Big|E_{(\mu_i, \theta_i)}[(X_i^2-\sigma^2)(Y_i^2-\sigma^2)\1(X_i^2\vee Y_i^2>\sigma^2\tau)]-\eta-\mu_i^2\theta_i^2\Big| \\
&\leq \frac{1}{n}\sum_{i=1}^n \Big[\min\{\mu_i^2, 3\sigma^2\tau\}\min\{\theta_i^2, 3\sigma^2\tau\} + 2\sigma^2\tau^{1/2}\phi(\tau^{1/2})\min\{\mu_i^2, 3\sigma^2\tau\} \\
&\qquad\qquad+ 2\sigma^2\tau^{1/2}\phi(\tau^{1/2})\min\{\theta_i^2, 3\sigma^2\tau\}\Big] \\
&\leq \frac{1}{n}\Big[\min\{qs^4, 3\sigma^2qs^2\tau, 9\sigma^4q\tau^2\} + 4\sigma^2\tau^{1/2}\phi(\tau^{1/2})\min\{ks^2, 3\sigma^2k\tau\}\Big],
\end{align*}
the last inequality follows from the fact that for $(\mu, \theta)\in\Omega(\beta, \epsilon, b)$, there are at most $k$ nonzero entries for either $\mu$ or $\theta$, and there are at most $q$ entries that are simultaneously nonzero for both $\mu$ and $\theta$.  

On the other hand, by Lemma~\ref{varestimator4}, for all $(\mu, \theta)\in \Omega(\beta, \epsilon, b)$ and $\tau\geq 1$, the variance of the estimator $\Qhat_4$ satisfies
\begin{align*}
&\Var_{(\mu, \theta)}(\Qhat_4) \\
&= \frac{1}{n^2}\sum_{i=1}^n\Var_{(\mu_i, \theta_i)}[(X_i^2-\sigma^2)(Y_i^2-\sigma^2)\1(X_i^2\vee Y_i^2>\sigma^2\tau)] \\
&\leq \frac{1}{n^2} \bigg[\sum_{i: \mu_i = \theta_i = 0}2d^{1/2}\tPhi(\tau^{1/2})^{1/2} \\
&\qquad\qquad + \sum_{i: \mu_i\neq 0\text{ or }\theta_i\neq 0}\big(4\sigma^2\mu_i^4\theta_i^2 + 4\sigma^2\mu_i^2\theta_i^4 + 16\sigma^4\mu_i^2\theta_i^2 + 2\sigma^4\mu_i^4 + 2\sigma^4\theta_i^4 \\
&\qquad\qquad\qquad\qquad\qquad\qquad + 8\sigma^6\mu_i^2 + 8\sigma^6\theta_i^2 + 4\sigma^8 + 8\sigma^4\mu_i^2\theta_i^2\tau^2\big) \bigg] \\
&\leq \frac{1}{n^2}\Big[2d^{1/2}n\tPhi(\tau^{1/2})^{1/2} + 8\sigma^2qs^6 + 16\sigma^4qs^4 + 4\sigma^4ks^4 + 16\sigma^6ks^2 + 8\sigma^8k + 8\sigma^4qs^4\tau^2\Big] \\
&\leq \frac{C}{n^2}\max\{n\tPhi(\tau^{1/2})^{1/2}, qs^4, qs^6, k, ks^2, ks^4, qs^4\tau^2\}.
\end{align*}
Again, the second to the last inequality follows from the fact that for $(\mu, \theta)\in\Omega(\beta, \epsilon, b)$, there are at most $k$ nonzero entries for either $\mu$ or $\theta$, and there are at most $q$ entries that are simultaneously nonzero for both $\mu$ and $\theta$.

Combining the bias and variance term, we have
\begin{align*}
&\sup_{(\mu, \theta)\in \Omega(\beta, \epsilon, b)}E_{(\mu, \theta)}(\Qhat_4-Q(\mu, \theta))^2 \\
&\leq \frac{C}{n^2}\Big[\min\{q^2s^8, q^2s^4\tau^2, q^2\tau^4\} + \tau\phi^2(\tau^{1/2})\min\{k^2s^4, k^2\tau^2\} \\
&\qquad\qquad + \max\{n\tPhi(\tau^{1/2})^{1/2}, qs^4, qs^6, k, ks^2, ks^4, qs^4\tau^2\}\Big] \\
&= \frac{C}{n^2}\Big[\min\{n^{2\epsilon+8b}, n^{2\epsilon+4b}\tau^2, n^{2\epsilon}\tau^4\} + \tau\phi^2(\tau^{1/2})\min\{n^{2\beta+4b}, n^{2\beta}\tau^2\} \\
&\qquad\qquad + \max\{n\tPhi(\tau^{1/2})^{1/2}, n^{\epsilon+4b}, n^{\epsilon+6b}, n^\beta, n^{\beta+2b}, n^{\beta+4b}, n^{\epsilon+4b}\tau^2\}\Big].
\end{align*}
Let $\tau = 4\log n$, then we have $\tPhi(\tau^{1/2}) \leq C\phi(\tau^{1/2}) = O(n^{-2})$ for some constant $C$.  It follows that for $b>0$,
\begin{align*}
\sup_{(\mu, \theta)\in \Omega(\beta, \epsilon, b)}E_{(\mu, \theta)}(\Qhat_4-Q(\mu, \theta))^2 &\leq C\max\Big\{n^{2\epsilon-2}(\log n)^4, n^{\epsilon+6b-2}, n^{\beta+4b-2}\Big\} \\
&\leq C\gamma_n(\beta, \epsilon, b),
\end{align*}
where for $\frac{\beta}{2}\leq\epsilon\leq\frac{3\beta}{4}$, 
\begin{equation}
\gamma_n(\beta, \epsilon, b) = \left\{
\begin{array}{ll}
n^{2\epsilon-2}(\log n)^4 &\text{if }0<b\leq\frac{2\epsilon-\beta}{4}, \\
n^{\beta+4b-2} &\text{if }\frac{2\epsilon-\beta}{4}<b\leq\frac{\beta-\epsilon}{2}, \\
n^{\epsilon+6b-2} &\text{if }b>\frac{\beta-\epsilon}{2}, \\
\end{array}\right. \label{Q4rated}
\end{equation}
and for $\frac{3\beta}{4}<\epsilon\leq\beta$,
\begin{equation}
\gamma_n(\beta, \epsilon, b) = \left\{
\begin{array}{ll}
n^{2\epsilon-2}(\log n)^4 &\text{if }0<b\leq\frac{\epsilon}{6}, \\
n^{\epsilon+6b-2} &\text{if }b>\frac{\epsilon}{6}. \label{Q4rateud}
\end{array}\right.
\end{equation}
Combining \eqref{Q0rateTwo}, \eqref{Q4rated} and \eqref{Q4rateud} matches the lower bounds in Theorem~\ref{two-inf-dense}.   

One can also check that when $\frac{\beta}{2}\leq\epsilon\leq\beta$ and $s = \sigma\sqrt{d\log n}$, 
\[\sup_{\substack{(\mu, \theta): \|\mu\|_0\leq k_n, \|\theta\|_0\leq k_n, \\\|\mu\|_\infty\leq \sigma\sqrt{d\log n}, \|\theta\|_\infty\leq \sigma\sqrt{d\log n}}}E_{(\mu, \theta)}(\Qhat_4-Q(\mu, \theta))^2 \leq Cn^{2\epsilon-2}(\log n)^4.\]
\end{proof}

\subsubsection{Proof of Upper Bound in Theorem~\ref{two-det-Q}}
We establish an upper bound for the hypothesis testing problem \eqref{twoTestingA} by constructing tests that asymptotically separate the null hypothesis from the alternative hypothesis over the detectable regions.  This, accompanied by the lower bound derived in Section~\ref{det-low-proof} that characterizes the undetectable regions, completes the proof of Theorem~\ref{two-det-Q}.  We divide the proof into two cases: the sparse regime where $0<\epsilon<\frac{\beta}{2}$, and the dense regime where $\frac{\beta}{2}\leq\epsilon\leq\beta$.

\begin{proof}[Proof for sparse regime]
Suppose that $0<\epsilon<\frac{\beta}{2}$ and $a\wedge b> 0$.  We will show that in this case, the test $\psi^* = \1(\Qhat_2\geq\lambda_n)$, where $\Qhat_2$ is defined in \eqref{estimator2} with $\tau = \log n$ and $\lambda_n = \half n^{\epsilon+2a+2b-1}$ (or, say, $\lambda_n = n^{2a\vee b-1}$) has sum of maximal type I error and maximal type II error goes to 0 as $n$ tends to infinity. 

Following similar variance and bias calculation as in the proof of Theorem~\ref{two-inf-sparse-up} in Section~\ref{sparse-upper-proof} while allowing for different signal strengths $r=n^a$ and $s=n^b$ for $\mu$ and $\theta$, we see that under $H_0$ where $q_n = 0$,
\begin{align*}
\sup_{(\mu, \theta)\in\Omega_0(\beta, a, b)}P_{(\mu, \theta)}(\Qhat_2\geq\lambda_n) &\leq \sup_{(\mu, \theta)\in\Omega_0(\beta, a, b)}\frac{\Var_{(\mu, \theta)}(\Qhat_2)}{\lambda_n^2} \\
&\leq C\frac{n^{\beta+4a\vee b-\half}\sqrt{\log n}+\log n}{n^{2\epsilon+4a+4b}} \longrightarrow 0.
\end{align*}
On the other hand, under $H_1$, since 
\begin{align*}
E_{(\mu, \theta)}(\Qhat_2) &\geq Q(\mu, \theta) - \frac{2}{n}\sum_{i=1}^n [\mu_i^2\min\{2\sigma^2\tau, \theta_i^2\} + \theta_i^2\min\{2\sigma^2\tau, \mu_i^2\}] \\
&\geq n^{\epsilon+2a+2b-1}-Cn^{\epsilon+2a\vee b-1}\log n,
\end{align*}
we have
\begin{align*}
\sup_{(\mu, \theta)\in\Omega_1(\beta, \epsilon, a, b)} P_{(\mu, \theta)}(\Qhat_2<\lambda_n) &\leq \sup_{(\mu, \theta)\in\Omega_1(\beta, \epsilon, a, b)}\frac{\Var_{(\mu, \theta)}(\Qhat_2)}{(E_{(\mu, \theta)}(\Qhat_2)-\lambda_n)^2} \\
&\leq C\frac{n^{\epsilon+4a\vee b+2a\wedge b}}{n^{2\epsilon+4a+4b}(1+o(1))} \longrightarrow 0.
\end{align*}
\end{proof}

\begin{proof}[Proof for dense regime]
Suppose that $0<\epsilon<\frac{\beta}{2}$, $a\wedge b>\frac{\beta-2\epsilon}{4}$ and $a\vee b>0$.  We will show that in this case, the test $\psi^* = \1(\Qhat_4\geq\lambda_n)$, where $\Qhat_4$ is as defined in \eqref{estimator4} with $\tau = 4\log n$ and $\lambda_n = \half n^{\epsilon+2a+2b-1}$ has sum of maximal type I error and maximal type II error goes to 0 as $n$ tends to infinity.

We follow similar variance and bias calculation as in the proof of Theorem~\ref{two-inf-dense-up} in Section~\ref{dense-upper-proof} while allowing for different signal strengths $r=n^a$ and $s=n^b$ for $\mu$ and $\theta$.  Under $H_0$, by Lemma~\ref{biasestimator4},
\[|E_{(\mu, \theta)}(\Qhat_4)| \leq \frac{C}{n}k\tau^{3/2}\phi(\tau^{1/2}) \leq Cn^{\beta-3}(\log n)^{3/2} = o(\lambda_n),\]
so when $n$ is sufficiently large, 
\begin{align*}
\sup_{(\mu, \theta)\in\Omega_0(\beta, a, b)}P_{(\mu, \theta)}(\Qhat_4\geq\lambda_n) &\leq \sup_{(\mu, \theta)\in\Omega_0(\beta, a, b)}\frac{\Var_{(\mu, \theta)}(\Qhat_4)}{(\lambda_n-E_{(\mu, \theta)}(\Qhat_4))^2} \\
&\leq C\frac{n^{\beta+4a\vee b}}{n^{2\epsilon+4a+4b}} \longrightarrow 0.
\end{align*}
Also, under $H_1$, since 
\begin{align*}
E_{(\mu, \theta)}(\Qhat_4) &\geq Q(\mu, \theta) - \frac{1}{n}\Big[\min\{qr^2s^2, 3\sigma^2qr^2\tau, 3\sigma^2qs^2\tau, 9\sigma^4q\tau^2\} \\
&\qquad\qquad\qquad\qquad + 2\sigma^2\tau^{1/2}\phi(\tau^{1/2})\min\{kr^2, 3\sigma^2k\tau\} \\
&\qquad\qquad\qquad\qquad + 2\sigma^2\tau^{1/2}\phi(\tau^{1/2})\min\{ks^2, 3\sigma^2k\tau\}\Big] \\
&= n^{\epsilon+2a+2b-1}(1+o(1)),
\end{align*}
we have
\begin{align*}
\sup_{(\mu, \theta)\in\Omega_1(\beta, \epsilon, a, b)} P_{(\mu, \theta)}(\Qhat_4<\lambda_n) &\leq \sup_{(\mu, \theta)\in\Omega_1(\beta, \epsilon, a, b)}\frac{\Var_{(\mu, \theta)}(\Qhat_4)}{(E_{(\mu, \theta)}(\Qhat_4)-\lambda_n)^2} \\
&\leq C\frac{n^{\epsilon+4a\vee b+2a\wedge b} + n^{\beta+4a\vee b}}{n^{2\epsilon+4a+4b}(1+o(1))} \longrightarrow 0.
\end{align*}
\end{proof}

\subsubsection{Proof of Lemma~\ref{varboundlemma}}
Let $B(\theta) = E(Y^2-\tau\sigma^2)_+ - \theta_0$.  We first note that $B(-\theta) = B(\theta)\geq 0$ for $\theta\geq 0$.  This follows from
\begin{align*}
B'(\theta) &= 2\sigma[\phi(\tau^{1/2}-\theta/\sigma)-\phi(\tau^{1/2}+\theta/\sigma)] \\
&\qquad - 2\theta[\Phi(\tau^{1/2}-\theta/\sigma)-\Phi(-\tau^{1/2}-\theta/\sigma)-1] \\
&\geq 0
\end{align*}
and $B(0) = 0$.  So we have $B(\theta) = E(Y^2-\tau\sigma^2)_+ - \theta_0\geq 0$ for all $\theta\in\R$.  It follows that $(E[(Y^2-\tau\sigma^2)_+ - \theta_0])^2 \leq (E(Y^2-\tau\sigma^2)_+)^2 \leq E[(Y^2-\tau\sigma^2)_+^2]$.  To bound the term $E[(Y^2-\tau\sigma^2)_+^2]$,  we consider two cases: $\theta\leq \sigma$ and $\theta>\sigma$.  It follows from the proof of Lemma 1 in \cite{CaiLow2005} that when $\theta\leq\sigma$, then
\[E[(Y^2-\tau\sigma^2)_+^2] \leq 6\sigma^2\theta^2 + \sigma^4\frac{4\tau^{1/2}+18}{e^{\tau/2}}.\]
On the other hand, when $\theta>\sigma$, we have
\[E[(Y^2-\tau\sigma^2)_+^2] \leq E[Y^4] = \theta^4 + 6\sigma^2\theta^2 + 3\sigma^4 \leq 10\theta^4.\]
If follows that we have
\[(E[(Y^2-\tau_n\sigma^2)_+-\theta_0])^2 \leq \max\bigg\{6\sigma^2\theta^2 + \sigma^4\frac{4\tau^{1/2}+18}{e^{\tau/2}}, 10\theta^4\bigg\}.\]

\subsubsection{Proof of Lemma~\ref{biasestimator4}}
The proof of Lemma~\ref{biasestimator4} is built on Lemma~\ref{e0} and Lemma~\ref{e1}.
 
\begin{Lemma}\label{e0}
Let $Y\sim N(\theta, \sigma^2)$.  Then for $\tau\geq 1$,
\begin{align*}
&E[(Y^2-\sigma^2)\1(Y^2\leq\sigma^2\tau)] \\
&= \theta^2\bigg[\tPhi(-\tau^{1/2}-\frac{\theta}{\sigma}\Big)-\tPhi\Big(\tau^{1/2}-\frac{\theta}{\sigma}\Big)\bigg] \\
&\qquad + \phi\bigg(\tau^{1/2}+\frac{\theta}{\sigma}\bigg)[-\sigma^2\tau^{1/2}+\sigma\theta] + \phi\bigg(\tau^{1/2}-\frac{\theta}{\sigma}\bigg)[-\sigma^2\tau^{1/2}-\sigma\theta].
\end{align*}
In particular, when $\theta=0$, 
\[E[(Y^2-\sigma^2)\1(Y^2\leq\sigma^2\tau)] = -2\sigma^2\tau^{1/2}\phi(\tau^{1/2}).\]
\end{Lemma}
\begin{proof}
Let $\lambda = \tau^{1/2}$.  We have
\begin{align*}
&E[Y^2\1(Y^2\leq\sigma^2\tau)] \\
&= \int_{-\sigma\lambda}^{\sigma\lambda}y^2\frac{1}{\sqrt{2\pi}\sigma}e^{-(y-\theta)^2/2\sigma^2}\ dy \\
&= \int_{-\lambda-\theta/\sigma}^{\lambda-\theta/\sigma}(\theta+\sigma z)^2\frac{1}{\sqrt{2\pi}}e^{-z^2/2}\ dz \\
&= \theta^2\int_{-\lambda-\theta/\sigma}^{\lambda-\theta/\sigma}\phi(z)\ dz + 2\sigma\theta\int_{-\lambda-\theta/\sigma}^{\lambda-\theta/\sigma}z\phi(z)\ dz + \sigma^2\int_{-\lambda-\theta/\sigma}^{\lambda-\theta/\sigma}z^2\phi(z)\ dz. 
\end{align*}
Using the fact that
\[\int_a^\infty \phi(z)\ dz = \tilde{\Phi}(a), \quad \int_a^\infty z\phi(z)\ dz = \phi(a), \quad \int_a^\infty z^2\phi(z)\ dz = a\phi(a) + \tilde{\Phi}(a),\]
we have
\begin{align*}
&E[Y^2\1(Y^2\leq\sigma^2\tau)] \\
&= \theta^2[\tPhi(-\lambda-\theta/\sigma)-\tPhi(\lambda-\theta/\sigma)] + 2\sigma\theta[\phi(-\lambda-\theta/\sigma)-\phi(\lambda-\theta/\sigma)] \\
&\qquad + \sigma^2[(-\lambda-\theta/\sigma)\phi(-\lambda-\theta/\sigma)+\tPhi(-\lambda-\theta/\sigma)\\
&\qquad\qquad\quad -(\lambda-\theta/\sigma)\phi(\lambda-\theta/\sigma)-\tPhi(\lambda-\theta/\sigma)] \\
&= (\theta^2+\sigma^2)[\tPhi(-\lambda-\theta/\sigma)-\tPhi(\lambda-\theta/\sigma)] + \phi(\lambda+\theta/\sigma)[-\sigma^2\lambda+\sigma\theta] \\
&\qquad+ \phi(\lambda-\theta/\sigma)[-\sigma^2\lambda-\sigma\theta],
\end{align*}
the last equality due to $\phi(-\lambda-\theta/\sigma) = \phi(\lambda+\theta/\sigma)$.  The proof is complete since $\sigma^2E[\1(Y^2<\sigma^2\tau)] = \sigma^2[\tPhi(-\lambda-\theta/\sigma)-\tPhi(\lambda-\theta/\sigma)]$.
\end{proof}

\begin{Lemma}\label{e1}
Let $Y\sim N(\theta, \sigma^2)$ and set $\theta_0 = E_0[(Y^2-\sigma^2)\1(Y^2\leq\sigma^2\tau)]$, where the expectation is taken under $\theta=0$.  Then for $\tau\geq 1$,
\[\big|E[(Y^2-\sigma^2)\1(Y^2\leq\sigma^2\tau)]-\theta_0\big| \leq \min\{\theta^2, 3\sigma^2\tau\}.\]
\end{Lemma}
\begin{proof}
Let $\lambda = \tau^{1/2}$.  It follows from Lemma~\ref{e0} that $\theta_0 = -2\sigma^2\lambda\phi(\lambda)$.  Set $B(\theta) = E[(Y^2-\sigma^2)\1(Y^2\leq\sigma^2\tau)]-\theta_0$, then 
\[E[(Y^2-\sigma^2)\1(Y^2\leq\sigma^2\tau)] \leq E[Y^2\1(Y^2\leq\sigma^2\tau)] \leq\sigma^2\lambda^2,\]
and
\begin{align*}
E[(Y^2-\sigma^2)\1(Y^2\leq\sigma^2\tau)] &= E(Y^2-\sigma^2) - E[(Y^2-\sigma^2)\1(Y^2>\sigma^2\tau)] \\
&\geq \theta^2-E(Y^2) = -\sigma^2 \geq-\sigma^2\lambda^2,
\end{align*}
hence
\[|B(\theta)| \leq \big|E[(Y^2-\sigma^2)\1(Y^2\leq\sigma^2\tau)]\big| + |\theta_0| \leq \sigma^2\lambda^2 + 2\sigma^2\lambda\phi(\lambda) \leq 3\sigma^2\lambda^2 = 3\sigma^2\tau.\]

Straightforward calculation yields for $\theta\geq 0$
\begin{align}
B'(\theta) &= \sigma(1+\lambda^2)[\phi(\lambda+\theta/\sigma)-\phi(\lambda-\theta/\sigma)] \nonumber\\
&\qquad + 2\theta[\tPhi(-\lambda-\theta/\sigma)-\tPhi(\lambda-\theta/\sigma)], \label{diff1}\\
B''(\theta) &= \phi(\lambda+\theta/\sigma)[-\lambda^2(\lambda+\theta/\sigma)-\lambda+\theta/\sigma] \nonumber\\
&\qquad + \phi(\lambda-\theta/\sigma)[-\lambda^2(\lambda-\theta/\sigma)-\lambda-\theta/\sigma] \nonumber\\
&\qquad + 2[\tPhi(-\lambda-\theta/\sigma)-\tPhi(\lambda-\theta/\sigma)]. \label{diff2}
\end{align}
It suffices to only consider $\theta\geq 0$ since $B(\theta) = B(-\theta)$.  It follows from \eqref{diff1} that for all $\theta\geq 0$, $B'(\theta)\leq 2\theta$.  Since $B(0) = 0$, it follows that
\[B(\theta) \leq \theta^2.\]
On the other hand, $\theta_0\leq 0$ immediately gives $B(\theta) \geq -\sigma^2 \geq -\theta^2$ for $\theta\geq\sigma$.  For $0\leq\theta<\sigma$, we have $\sigma(1+\lambda^2)\geq 2\theta$.  For $x>0$, we have $\tPhi(x)<x^{-1}\phi(x)$, so $\tPhi(-\lambda-\theta/\sigma) = 1 - \tPhi(\lambda+\theta/\sigma) \geq 1-(\lambda+\theta/\sigma)^{-1}\phi(\lambda+\theta/\sigma)$.  It then follows from \eqref{diff1} that for $0\leq\theta<\sigma$,
\begin{align*}
B'(\theta) &\geq 2\theta[\phi(\lambda+\theta/\sigma)-\phi(\lambda-\theta/\sigma)+\tPhi(-\lambda-\theta/\sigma)-\tPhi(\lambda-\theta/\sigma)] \\
&\geq 2\theta[1+(1-(\lambda+\theta/\sigma)^{-1})\phi(\lambda+\theta/\sigma) - \phi(\lambda-\theta/\sigma) - \tPhi(\lambda-\theta/\sigma)] \\
&\geq 2\theta\bigg[1+(1-(\lambda+\theta/\sigma)^{-1})\phi(\lambda+\theta/\sigma) - \frac{1}{\sqrt{2\pi}}-\half\bigg] \\
&\geq 0.
\end{align*}
Coupled with $B(0) = 0$, this implies that $B(\theta)\geq 0\geq -\theta^2$ for $0\leq\theta<\sigma$.  Hence,
\[B(\theta) \geq -\theta^2.\]
Therefore, for all $\theta$, we have
\[|B(\theta)| \leq \theta^2.\]
\end{proof}

\begin{proof}[Proof of Lemma~\ref{biasestimator4}]
Let $\theta_0 = E_0[(Y^2-\sigma^2)\1(Y^2\leq\sigma^2\tau)] = -2\sigma^2\tau^{1/2}\phi(\tau^{1/2})$, the second equality due to Lemma~\ref{e0}.  It follows from the expression
\begin{align*}
&E[(X^2-\sigma^2)(Y^2-\sigma^2)\1(X^2\vee Y^2>\sigma^2\tau)] \\
&\qquad = \mu^2\theta^2-E[(X^2-\sigma^2)\1(X^2\leq\sigma^2\tau)]E[(Y^2-\sigma^2)\1(Y^2\leq\sigma^2\tau)]
\end{align*}
and
\begin{align*}
\eta &= E_{(0, 0)}[(X^2-\sigma^2)(Y^2-\sigma^2)\1(X^2\vee Y^2>\sigma^2\tau)] \\
&= - E_0[(X^2-\sigma^2)\1(X^2\leq\sigma^2\tau)]E_0[(Y^2-\sigma^2)\1(Y^2\leq\sigma^2\tau)] \\
&= -\theta_0^2
\end{align*}
that we have
\begin{align}
&\big|E[(X^2-\sigma^2)(Y^2-\sigma^2)\1(X^2\vee Y^2>\sigma^2\tau)]-\eta-\mu^2\theta^2\big| \label{bias}\\
&= \big|E[(X^2-\sigma^2)\1(X^2\leq\sigma^2\tau)]E[(Y^2-\sigma^2)\1(Y^2\leq\sigma^2\tau)] - \theta_0^2\big|. \nonumber
\end{align}
Using the decomposition
\[AB-ab = (A-a)(B-b) + a(B-b) + b(A-a),\]
and triangle inequality, we get
\begin{align*}
&\big|E[(X^2-\sigma^2)\1(X^2\leq\sigma^2\tau)]E[(Y^2-\sigma^2)\1(Y^2\leq\sigma^2\tau)] - \theta_0^2\big| \\
&\leq \big|E[(X^2-\sigma^2)\1(X^2\leq\sigma^2\tau)]-\theta_0\big|\big|E[(Y^2-\sigma^2)\1(Y^2\leq\sigma^2\tau)] - \theta_0\big| \\
&\qquad + \big|\theta_0\big|\big|E[(X^2-\sigma^2)\1(X^2\leq\sigma^2\tau)]-\theta_0\big| + \big|\theta_0\big|\big|E[(Y^2-\sigma^2)\1(Y^2\leq\sigma^2\tau)]-\theta_0\big| \\
&\leq \min\{\mu^2, 3\sigma^2\tau\}\min\{\theta^2, 3\sigma^2\tau\} + 2\sigma^2\tau^{1/2}\phi(\tau^{1/2})\min\{\mu^2, 3\sigma^2\tau\} \\
&\qquad+ 2\sigma^2\tau^{1/2}\phi(\tau^{1/2})\min\{\theta^2, 3\sigma^2\tau\},
\end{align*}
the last inequality follows from Lemma~\ref{e1} and substitution of value of $\theta_0$.
\end{proof}

\subsubsection{Proofs of Lemma~\ref{varestimator4}}
We have
\begin{align*}
&\Var[(X^2-\sigma^2)(Y^2-\sigma^2)\1(X^2\vee Y^2>\sigma^2\tau)] \\
&= E[(X^2-\sigma^2)^2(Y^2-\sigma^2)^2\1(X^2\vee Y^2>\sigma^2\tau)] \\
&\qquad- \big\{E[(X^2-\sigma^2)(Y^2-\sigma^2)\1(X^2\vee Y^2>\sigma^2\tau)]\big\}^2 \\
&= E[(X^2-\sigma^2)^2(Y^2-\sigma^2)^2] - E[(X^2-\sigma^2)^2\1(X^2\leq\sigma^2\tau)(Y^2-\sigma^2)^2\1(Y^2\leq\sigma^2\tau)] \\
&\qquad - \big\{E[(X^2-\sigma^2)(Y^2-\sigma^2)]-E[(X^2-\sigma^2)\1(X^2\leq\sigma^2\tau)(Y^2-\sigma^2)\1(Y^2\leq\sigma^2\tau)]\big\}^2 \\
&= \Var[(X^2-\sigma^2)(Y^2-\sigma^2)] - E[(X^2-\sigma^2)^2\1(X^2\leq\sigma^2\tau)]E[(Y^2-\sigma^2)^2\1(Y^2\leq\sigma^2\tau)] \\
&\qquad - \big\{E[(X^2-\sigma^2)\1(X^2\leq\sigma^2\tau)]E[(Y^2-\sigma^2)\1(Y^2\leq\sigma^2\tau)]\big\}^2\\
&\qquad +2\mu^2\theta^2E[(X^2-\sigma^2)\1(X^2\leq\sigma^2\tau)]E[(Y^2-\sigma^2)\1(Y^2\leq\sigma^2\tau)] \\
&\leq \Var[(X^2-\sigma^2)(Y^2-\sigma^2)] \\
&\qquad +2\mu^2\theta^2E[(X^2-\sigma^2)\1(X^2\leq\sigma^2\tau)]E[(Y^2-\sigma^2)\1(Y^2\leq\sigma^2\tau)] \\
&\leq \Var[(X^2-\sigma^2)(Y^2-\sigma^2)] + 8\sigma^4\mu^2\theta^2\tau^2.
\end{align*}
Straightforward calculation yields
\begin{align*}
&\Var[(X^2-\sigma^2)(Y^2-\sigma^2)] \\
&= \Var(X^2-\sigma^2)\Var(Y^2-\sigma^2) + [E(X^2-\sigma^2)]^2\Var(Y^2-\sigma^2) \\
&\qquad + \Var(X^2-\sigma^2)[E(Y^2-\sigma^2)]^2 \\
&= [4\sigma^2\mu^2+2\sigma^4][4\sigma^2\theta^2+2\sigma^4] + \mu^4[4\sigma^2\theta^2+2\sigma^4] + \theta^4[4\sigma^2\mu^2+2\sigma^4] \\
&= 4\sigma^2\mu^4\theta^2 + 4\sigma^2\mu^2\theta^4 + 16\sigma^4\mu^2\theta^2 + 2\sigma^4\mu^4 + 2\sigma^4\theta^4 + 8\sigma^6\mu^2 + 8\sigma^6\theta^2 + 4\sigma^8.
\end{align*}
When $\mu=\theta=0$, let $d = E_{(0, 0)}[(X^2-\sigma^2)^4(Y^2-\sigma^2)^4] < \infty$, and we have
\begin{align*}
&\Var[(X^2-\sigma^2)(Y^2-\sigma^2)\1(X^2\vee Y^2>\sigma^2\tau)] \\
&\leq E[(X^2-\sigma^2)^2(Y^2-\sigma^2)^2\1(X^2\vee Y^2>\sigma^2\tau)] \\
&\leq \Big(E[(X^2-\sigma^2)^4(Y^2-\sigma^2)^4]P(X^2\vee Y^2>\sigma^2\tau)\Big)^{1/2} \\
&= d^{1/2}\Big(1-P(|Z|\leq\tau^{1/2})^2\Big)^{1/2}, \qquad\text{where }Z\sim N(0, 1) \\
&\leq (2d)^{1/2}\Big(1-P(|Z|\leq\tau^{1/2})\Big)^{1/2} \\
&= 2d^{1/2}\tPhi(\tau^{1/2})^{1/2},
\end{align*}
the second inequality follows from the Cauchy-Schwarz inequality.
\end{appendices}
\end{document}